\documentclass[12pt]{article}
\usepackage[left=.9in, right=.9in, top=1in, bottom=1in]
{geometry}
\usepackage[english]{babel}
\usepackage{amsfonts,amssymb,graphicx,amsmath,amsthm, float}
\usepackage{comment,xcolor}
\usepackage{graphicx}
\usepackage[shortlabels]{enumitem}\usepackage{epstopdf}
\usepackage{authblk} 
\usepackage[colorinlistoftodos]{todonotes}
\theoremstyle{plain}
\newtheorem{theorem}{Theorem}[section]

\newtheorem{corollary}[theorem]{Corollary}

\newcommand{\pr}{\mathfrak{pr}}

\theoremstyle{remark}

\usepackage{array}
\newcolumntype{L}[1]{>{\raggedright\let\newline\\\arraybackslash\hspace{0pt}}m{#1}}
\newcolumntype{C}[1]{>{\centering\let\newline\\\arraybackslash\hspace{0pt}}m{#1}}
\newcolumntype{R}[1]{>{\raggedleft\let\newline\\\arraybackslash\hspace{0pt}}m{#1}}

\usepackage{color}
\definecolor{red}{rgb}{1,0,0}

\definecolor{blu}{rgb}{0,0,1}

\usepackage{pgf,tikz}
\usepackage{mathrsfs}
\usetikzlibrary{arrows}
\definecolor{qqqqff}{rgb}{0.,0.,1.}

\title{Total Prime Labelings of Various Graphs}

\author[1]{N.~Bradley Fox$^*$}
\author[2]{Joseph Spaeth}
\affil[1]{Department of Mathematics and Statistics, Austin Peay State University, Clarksville, TN, USA
}
\affil[2]{Department of Mathematics and Statistics, Austin Peay State University, Clarksville, TN, USA}

\date{\today}

\begin{document}

\maketitle

\begin{abstract}
A total prime labeling of a graph of order $n$ is an extension of a prime labeling in which we distinctly label the vertices and edges. The goal of the labeling is for adjacent vertex labels to be relatively prime, and for each vertex of degree at least two, the greatest common divisor of the labels on its incident edges is equal to 1. In this paper, we construct total prime labelings by extending known prime and minimum coprime labelings and by developing new constructions for various classes of graphs. In particular, we show that snakes, books, prisms, prime trees, certain families of windmills, and other families of graphs are total prime.
\end{abstract}


\section{Introduction}\label{sec: intro}

A total prime labeling is an extension of a prime labeling, which was conceived by Entringer and first introduced by Tout, Dabboucy, and Howalla~\cite{TDH}. Given a simple graph $G$ of order~$n$, a \textit{prime labeling} is a bijective function assigning the integers $1,2,\ldots, n$ to the vertices such that every pair of labels on adjacent vertices is relatively prime. A graph is called \textit{prime} if such a labeling exists. In the last forty years, prime labelings have been constructed or shown to not exist for many classes of graphs, as detailed in Gallian's dynamic survey on graph labelings~\cite{Gallian}.

A recent focus has been on variations of prime labelings, in which the part of the graph being labeled, the allowable set of labels, or the relatively prime condition has been altered. Examples include the neighborhood-prime labeling~\cite{PS}, minimum coprime labeling~\cite{AF,BDH}, vertex prime labeling~\cite{DLM}, and edge vertex prime labeling~\cite{JB}. The prime labeling varient that we will investigate is the \textit{total prime labeling}, introduced by Ramasubramanian and Kala~\cite{RK}, which involves labeling both the vertices and edges. For a graph $G$ with $|E|=m$ edges and $|V|=n$ vertices, a bijection $\ell: V\cup E\rightarrow \{1,2,3,\ldots, m+n\}$ is a total prime labeling if
\begin{itemize}
\item for each pair of adjacent vertices $u$ and $v$, the labels $\ell(u)$ and $\ell(v)$ are relatively prime, or $\gcd(\ell(u),\ell(v))=1$,
\item for each vertex $v$ of degree at least 2, the greatest common divisor of the labels of all incident edges is 1, or $\gcd(\ell(uv)| u\in N(v))=1$. 
\end{itemize}
A graph that admits a total prime labeling is referred to as \textit{total prime}.

Ramasubramanian and Kala first proved all paths, stars, combs, friendship graphs, and fans to be total prime~\cite{RK}. Furthermore, they showed that cycles $C_n$ are total prime for even $n$, but are not total prime when $n$ is odd. In fact, odd cycles are the only graphs proven thus far not to be total prime. Meena and Ezhil~\cite{ME1} investigated wheels, gears, double combs, and triangular books $B_3^n$ for even $n$, showing that each admits a total prime labeling. Their later works include total prime labelings for triangular snakes $t_n$ when $n$ is even~\cite{ME2} and rectangular books $B_4^n$~\cite{ME3}. 
We continue the investigation of total prime labelings for a variety of classes of graphs. Our results are organized based on whether the graphs are known to be prime, beginning in the next section with prime graphs that contain at least one cycle. In Section~\ref{sec: nonprime} we focus on graphs that are not prime, but do exhibit total prime labelings. Section~\ref{sec: other} features classes conjectured to be prime but whose primality is not fully resolved. We consider classes of trees in Section~\ref{sec: trees}. Graphs that do not have total prime labelings are investigated in Section~\ref{sec: non-TPL}. Finally, we discuss several open problems in Section~\ref{sec: open}.

\section{Prime Cyclic Graphs}\label{sec: cyclic}

We begin our investigation of total prime labelings with prime graphs that include a cycle, saving the case of trees for later in the paper.  We first examine the \textit{helm graph}, which is based on the wheel graph $W_n = C_n + K_1$, defined by connecting one central vertex to every vertex of the outer cycle $C_n$. The helm graph, denoted $H_n$ where $n \geq 3$, is then formed from $W_n$ by adding a pendant edge to each vertex of the outer $n$-cycle. In \cite{RK}, a labeling was given to prove that the helm is total prime for all $n$, but their construction fails in the case of $n\equiv 1\pmod{3}$, because two adjacent vertices share a common factor of 3. We construct a new labeling that meets the criteria of a total prime labeling for all $n \geq 3$. An example of the helm $H_4$ with a total prime labeling is shown in Figure~\ref{helmgraph}.

\begin{figure}[h]
    \centering
\includegraphics[scale=0.35]{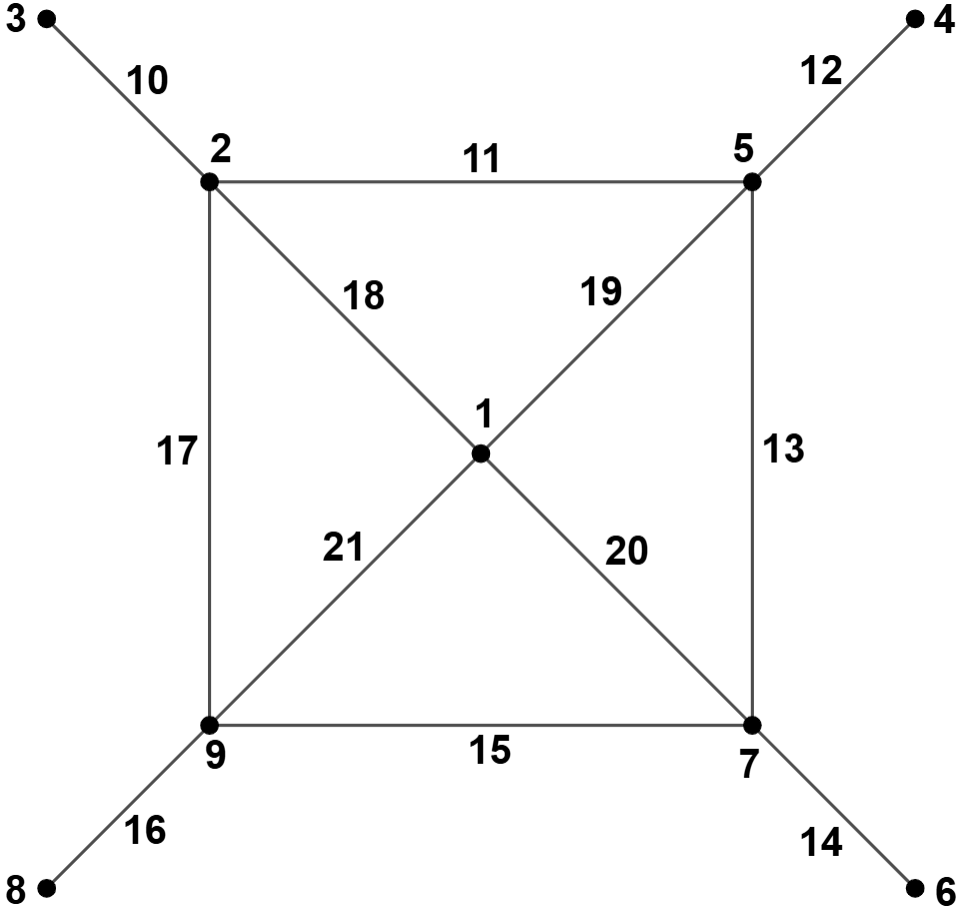}
    \caption{The helm $H_4$ with a total prime labeling}
    \label{helmgraph}
\end{figure}
\begin{theorem}\label{helm}
    The helm graph $H_n$ is total prime for all $n\geq 3$.
\end{theorem}
\begin{proof}
    The graph $H_n$ has $2n + 1$ vertices and $3n$ edges. Let $x$ be the center vertex, $w_1, w_2, \dots, w_n$ be the vertices that form the cycle, and $v_1, v_2, \ldots, v_n$ be their corresponding pendant vertices. We create a bijective labeling $\ell: V \cup E \rightarrow \{1, 2, \ldots, 5n + 1\}$ as follows:
    \begin{align*}
        \ell(x) &= 1, \\
        \ell(w_1) &= 2, \\
        \ell(v_1) &= 3, \\
        \ell(v_i) &= 2i \;\text{for } i = 2, 3, \ldots, n, \\
        \ell(w_i) &= 2i + 1 \;\text{for } i = 2, 3, \ldots, n, \\
        \ell(v_iw_i) &= 2i + 2n \;\text{for } i = 1, 2, \ldots, n, \\
        \ell(w_iw_{i+1}) &= 2i + 2n + 1 \;\text{for } i = 1, 2, \ldots, n-1, \\
        \ell(w_nw_1) &= 4n + 1, \\
        \ell(xw_i) &= 4n + i + 1 \;\text{for } i = 1, 2, \ldots, n.
    \end{align*}

    For $\ell$ to be a total prime labeling, we must have that the labels of every pair of adjacent vertices are relatively prime. First, since $\ell(x) = 1$, it is clear that $\gcd(\ell(x), \ell(w_i)) = 1$. Also, every pair $v_i$ and $w_i$ is labeled consecutively, so $\gcd(\ell(v_i), \ell(w_i)) = 1$. Finally, notice that the pairs of vertices on the cycle, $w_i$ and $w_{i+1}$, are labeled with consecutive odd integers with the exception of $\ell(w_1) = 2$, which is adjacent to odd labels $\ell(w_2) = 5$ and $\ell(w_n) = 2n + 1$. In either case, the pairs are relatively prime.

    It must also be true that for each vertex of degree at least 2, the labels of all incident edges have a greatest common divisor of 1. Note that each $v_i$ has degree 1, so we must only consider each $w_i$ and $x$. For $i = 1, 2, \ldots, n-1$, $w_i$ is incident on edges $v_iw_i$ and $w_iw_{i+1}$ that are labeled consecutively by $2i + 2n$ and $2i + 2n + 1$, respectively. This implies that the $\gcd(\ell(uw_i)|u\in N(w_i)) = 1$. When $i = n$, $w_n$ has incident edges $v_nw_n$ and $w_nw_1$ where $\gcd(\ell(v_nw_n), \ell(w_nw_1)) = \gcd(4n, 4n + 1) = 1$. Lastly, the edges $xw_i$ are labeled consecutively for all $i = 1, 2, \ldots, n$, so the labels of all incident edges of vertex $x$ have a gcd of 1.

    The bijection $\ell$ meets the criteria of a total prime labeling, so $H_n$ is total prime for all $n \geq 3$.
\end{proof}

We now develop a total prime labeling for a cycle with any chord connecting two non-adjacent vertices. Even though the cycle $C_n$ is only total prime if $n$ is even, we will show that it is always total prime once a chord is introduced, which we denote by $C_n^+$. See Figure~\ref{cyclechordgraph} for the example $C_9^+$. This graph class will be particularly important in upcoming results for graphs that have $C_n^+$ as a subgraph.

\begin{figure}[h]
    \centering
\includegraphics[scale=0.4]{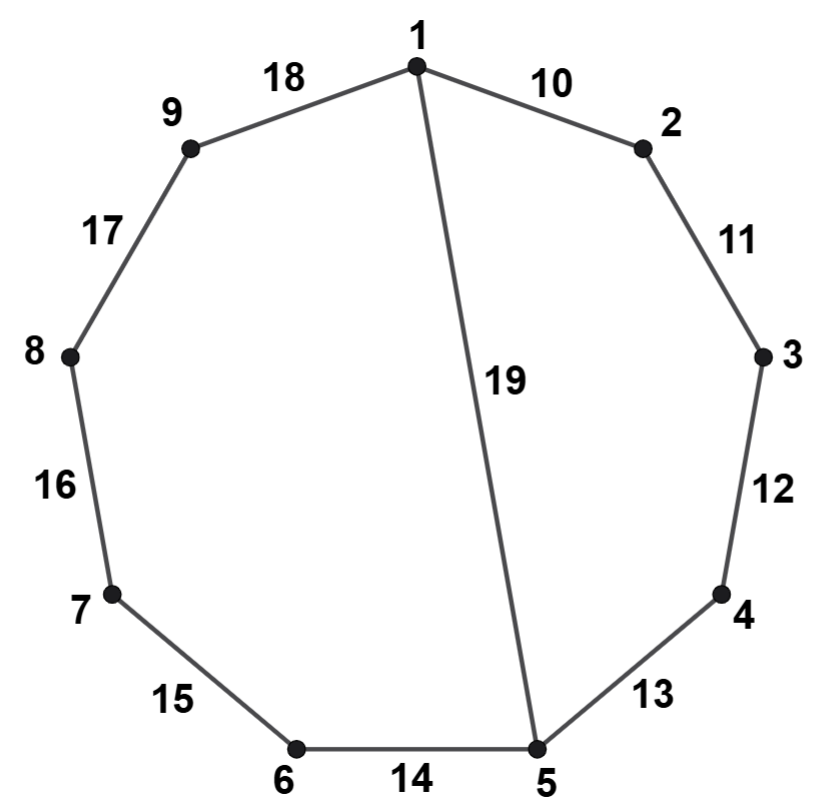}
    \caption{$C_9^+$ with a total prime labeling}
    \label{cyclechordgraph}
\end{figure}

\begin{theorem}\label{cyclechord}
    The graph consisting of a cycle with a chord, $C_n^+$, is total prime for all $n\geq 4$.
\end{theorem}
\begin{proof}
    The graph $C_n^+$ has $n$ vertices and $n+1$ edges. 
    Let $v_1, v_2, \ldots, v_n$ be the vertices of $C_n^+$, where $v_1$ and $v_k$ for some $k$ with $2<k<n$ are the endpoints of the chord.

    We define a labeling $\ell:V \cup E \rightarrow \{1, 2, \ldots, 2n+1\}$ in which
    \begin{align*}
        \ell(v_i) &= i \text{ for } i = 1, 2, \ldots, n, \\
        \ell(v_iv_{i+1}) &= n+i \text{ for } i = 1, 2, \ldots, n-1, \\
        \ell(v_nv_1) & =2n,\\
        \ell(v_1v_k) &= 2n+1.
    \end{align*}
All pairs of adjacent vertices are labeled consecutively, except for $v_1$ and its neighbors $v_k$ and $v_n$, where $\ell(v_1)=1$. Therefore, the labels of all pairs of adjacent vertices are relatively prime.

For any vertex $v_i$ with $i\neq 1,k$, its incident edges are labeled consecutively. Considering the labels of the edges incident on $v_1$, we have $\gcd(\ell(v_1v_2),\ell(v_nv_1),\ell(v_1v_k))=\gcd(n+1,2n,2n+1)=1$ since two labels are consecutive integers. Likewise for $v_k$, 
    $\gcd(\ell(v_{k-1}v_{k}),\ell(v_kv_{k+1}),\ell(v_1v_k))=\gcd(n+k-1,n+k,2n+1)=1$. Hence the gcd of the labels of the incident edges is 1 for all vertices, making $\ell$ a total prime labeling.
\end{proof}

We now consider a more general class of graphs, specifically Hamiltonian graphs, which by definition have a cycle containing all vertices. A condition is assumed for the maximum degree, denoted by $\Delta$, to ensure that the graph also contains a chord so that it has $C_n^+$ as a spanning subgraph.

\begin{theorem}\label{Hamiltonian}
    Let $G$ be a Hamiltonian graph with $\Delta\geq 3$. If $G$ is prime, then $G$ is total prime.
\end{theorem}
\begin{proof}
    Assume $G$ is of order $n$ with $m$ edges. Let $f: V \rightarrow \{1, 2, \ldots, n\}$ be the prime labeling on the vertices of $G$. We extend $f$ to create a labeling $\ell: V \cup E \rightarrow \{1, 2, \ldots, m+n\}$ in which $\ell(v) = f(v)$ for all $v\in V$.

    Since $G$ is Hamiltonian with at least one vertex having degree 3 or more, there exists a cycle $(v_1,v_2,\ldots, v_n,v_1)$ with a chord $v_1v_k$ for some $2<k< n$. We proceed by assigning edge labels for this cycle as was done in Theorem~\ref{cyclechord} using the labels $n+1,n+2,\ldots, 2n+1$. The remaining edges of $G$ can be arbitrarily assigned the labels within $\{2n+2,2n+3,\ldots, m+n\}$.

    We have that each pair of adjacent vertex labels is relatively prime since $f$ is a prime labeling. For each vertex $v_i$, $\gcd(\ell(uv_i)| u\in N(v_i))=1$ since two of the edge labels incident on $v_i$ from the cycle and chord subgraph are labeled consecutively. Thus, $\ell$ is a total prime labeling of $G$.
\end{proof}

\begin{corollary}\label{HamiltonCor}
    The following classes of graphs are total prime:
    \begin{enumerate}[(a)]
        \item The ladder graph $L_n$ for all lengths $n$
        \item The grid graph $P_m\square P_n$ where $n$ is prime, $m$ is even, and $3<m\leq n$
        \item The grid graph $P_{n+1}\square P_{n+1}$ where $n$ is an odd prime, $n=5$ or $n\equiv 3$ or $9\pmod{10}$, and $(n+1)^2+1$ is prime
        \item The star $(m,n)$-gon $S_n^{(m)}$ for all $m,n\geq 3$.
    \end{enumerate}
\end{corollary}
\begin{proof}
    Ladder graphs and star $(m,n)$-gons are Hamiltonian, as are grid graphs $P_m\square P_n$ when at least one of $m$ or $n$ is even.
    Therefore, each class is total prime by Theorem~\ref{Hamiltonian} 
    based on the following prime labeling results: (a)~\cite{Dean}, (b)~\cite{SPS}, (c)~\cite{Kanetkar}, and $(d)$~\cite{SY}.
\end{proof}

Note that prism graphs are conjectured to be prime when the cycle length is even, with many cases proven to be so~\cite{HLYZ}. Since this graph class is Hamiltonian, certain prisms could have been included in the previous corollary. However, we will consider all cases of prisms in Section~\ref{sec: other}.

Next, we turn our attention to the \textit{snake graph}, denoted $S_{k,n}$, which consists of $n$ cycles $C_k$ attached so that one edge from each cycle forms a path $P_{n+1}$. An example of a total prime snake graph $S_{5,3}$ is given in Figure~\ref{snakegraph}. Previous work has established that the snake graph is total prime when $k = 3$~\cite{ME2}, and we generalize this result by providing a total prime labeling for any $k\geq 3$ and $n\geq 2$. Note that when $n=1$, $S_{k,n}=C_k$, which is not total prime for odd $k$. The proof that follows makes use of two properties of the greatest common divisor for integers $a, b,$ and $t$:
\begin{align}
    \gcd(a, b) &= \gcd(a, b-a), \label{gcd1} \\
    \gcd(a, b) &= \gcd(a+tb, b). \label{gcd2}
\end{align}

\begin{figure}[h]
    \centering
\includegraphics[scale=0.42]{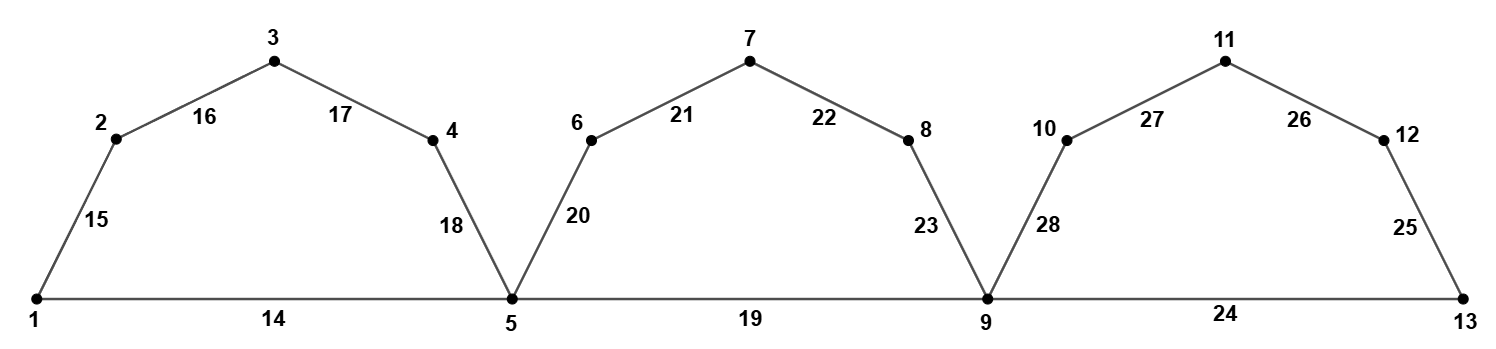}
    \caption{The snake $S_{5, 3}$ with a total prime labeling}
    \label{snakegraph}
\end{figure}
\begin{theorem}
    The snake graph $S_{k,n}$ is total prime for all $k\geq 3$ and $n\geq 2$.
\end{theorem}
\begin{proof}
    This graph has $n(k-1)+1$ vertices and $nk$ edges. We create a labeling $\ell: V\cup E \rightarrow \{1, 2, \ldots, 2nk-n+1\}$. Let $v_1$, $v_2, \ldots, v_{n+1}$ be the vertices along the path, and denote the vertices of the $i$th cycle as $v_i, w_{i, 1}, w_{i, 2}, \ldots, w_{i, k-2}, v_{i+1}$, with $1 \leq i \leq n$.
    
    We first label the vertices starting with $\ell(v_1) = 1$, then sequentially label clockwise along each of the $n$ cycles. That is, label each segment of vertices $w_{i, 1}, \ldots, w_{i, k-2}, v_{i+1}$ as $(i-1)(k-1)+2 , (i-1)(k-1)+3, \ldots, i(k-1)+1$.

    With the integers $n(k-1)+2, n(k-1)+3, \ldots, 2nk-n+1$, we sequentially label the edges of the first $n-1$ cycles clockwise along $v_iv_{i+1},v_{i}w_{i,1},\ldots, w_{i,k-2}v_{i+1}$ from $i=1$ to $n-1$. For the $n$th cycle, we label the edges counterclockwise as $v_nv_{n+1}, v_{n+1}w_{n,k-2},\ldots, v_nw_{n, 1}$.

    All adjacent vertex pairs other than $v_i$,$v_{i+1}$ are consecutive. Note that $\ell(v_i)=(i-1)(k-1)+1$, so for these remaining vertex pairs, we apply Equations~(\ref{gcd1}) and~(\ref{gcd2}) to obtain
    \begin{align*}
        \gcd(\ell(v_i), \ell(v_{i+1})) &= \gcd((i-1)(k-1)+1, i(k-1)+1) \\
        &= \gcd((i-1)(k-1)+1, k-1) \\
        &= \gcd(1, k-1) \\
        &= 1.
    \end{align*}

    For all degree-2 vertices, the incident edges are labeled consecutively, including those incident on $v_{n+1}$ due to the reverse direction on the last cycle. For vertices $v_2, \ldots, v_n$, the edges $w_{i-1, k-2}v_i$ and $v_iv_{i+1}$ are labeled consecutively, so $$\gcd(\ell(w_{i-1,k-2}v_i), \ell(v_iv_{i+1}), \ell(v_{i-1}v_i), \ell(v_iw_{i,1})) = 1.$$
    
    All pairs of adjacent vertices have relatively prime labels and the gcd of the labels of incident edges is 1 for all vertices. Thus, $S_{k,n}$ is total prime.

\end{proof}

The \textit{book graph}, denoted $B_k^n$ for $k \geq 3$ and $n \geq 2$, consists of $n$ cycles of length $k$ that all share exactly one common edge, and thus two common vertices of degree $n + 1$. Visually, the shared edge can be thought of as the spine of a book with each of the $n$ cycles being a page. See Figure~\ref{bookgraph} for the example $B_5^3$ with a total prime labeling.

Previous work has shown that books are total prime in certain cases of $k$ and $n$. In particular, the triangular book $B_3^n$ is shown to be total prime for even values of $n$ in~\cite{ME1}. A total prime labeling for the rectangular book $B_4^n$ is given in~\cite{ME3}. Here, we broaden these results by showing that $B_k^n$ is total prime for all $k \geq 3$ and $n \geq 3$. Note that when $n = 2$, the book graph $B_k^2$ takes the form of the cycle and chord $C_{2k-2}^+$ and is therefore total prime by Theorem~\ref{cyclechord}.

\begin{figure}[h]
    \centering
\includegraphics[scale=0.35]{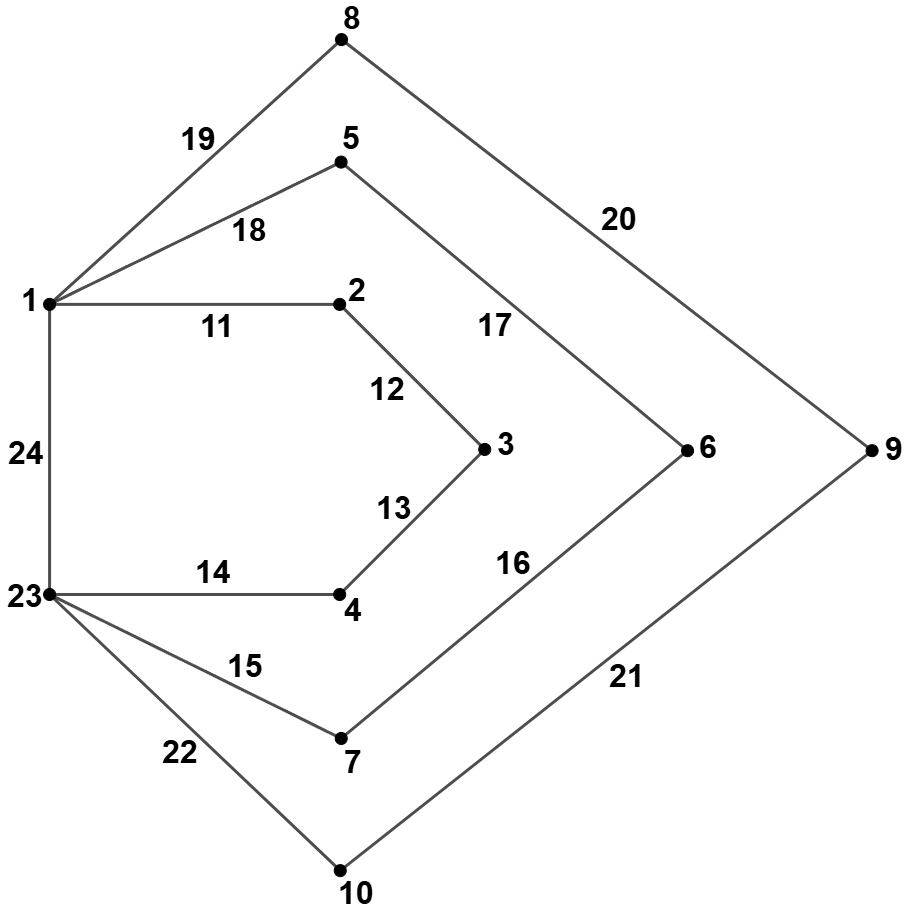}
    \caption{The book $B_5^3$ with a total prime labeling}
    \label{bookgraph}
\end{figure}
\begin{theorem} 
    The book graph $B_k^n$ is total prime for all $k \geq 3$ and $n \geq 3$.
\end{theorem}
\begin{proof}
    The graph $B_k^n$ has $n(k-2) + 2$ vertices and $n(k-1) + 1$ edges. Let $u$ and $v$ be the two vertices with degree $n+1$ that are shared by each page. The $i$th page, for $i=1,2,\ldots, n$, is the cycle $(u,x_{i,1},x_{i,2},\ldots, x_{i,k-2}, v, u)$. To construct a total prime labeling $\ell: V \cup E \rightarrow \{1, 2, \ldots, 2nk - 3n + 3\}$, we consider even and odd values of $k$ separately.

    \textit{Case 1}: Assume $k \geq 4$ is even. First, define $\ell(u) = 2$ and $\ell(v) = 1$. Then, let $\ell(x_{i,j}) = (k-2)(i-1) + j + 2$ for $i = 1, 2, \ldots, n$ and $j = 1, 2, \ldots, k-2$. To verify that all adjacent vertex labels are relatively prime, notice that the pairs of adjacent vertices $v,x_{i,k-2}$ and $u,v$ have relatively prime labels because $\ell(v) = 1$. Also, for the pair $u,x_{i,1}$, observe that $\ell(x_{i,1})$ is always odd. Since $\ell(u) = 2$, $\gcd(\ell(u), \ell(x_{i,1})) = 1$. All remaining pairs of adjacent vertices have relatively prime labels because they are labeled consecutively.

    To assign labels to the edges, consecutively label $ux_{1,1}$ to $x_{1, k-2}v$ with values $n(k-2) + 3, n(k-2) + 4, \ldots, (n+1)(k-2) + 3$. Then, consecutively label edges $vx_{2, k-2}$ to $x_{2,1}u$ with values $n(k-2) + k + 2, n(k-2) + k + 3, \ldots, (n+1)(k-2) + k + 2$. Continue this consecutive labeling pattern, alternating direction with each page, until the value $2nk - 3n + 2$ is assigned. At this point, the only edge remaining to be labeled is $uv$, where we assign $\ell(uv) = 2nk - 3n + 3$. Notice that each vertex $x_{i,j}$ has a degree of 2, with its two incident edges labeled consecutively. The incident edges for the vertex $u$ include $ux_{2,1}$ and $ux_{3,1}$, which are labeled consecutively. Because we assumed $n \geq 3$, we are assured that vertex $x_{3,1}$ exists. Lastly, vertex $v$ has consecutively labeled edges $vx_{1, k-2}$ and $vx_{2, k-2}$. Every vertex has at least one pair of consecutively labeled incident edges, so the labels of each set of incident edges have a gcd of 1 for all vertices.

    \textit{Case 2}: Assume $k \geq 3$ is odd. To label the vertices, first let $p = \max\{x \in \{1, 2, \ldots, 2nk - 3n + 3\} | x \text{ is prime}\}$. By Bertrand's postulate, this largest prime $p$ satisfies $p>(2nk-3n+3)/2$, so it is relatively prime with all other integers in the labeling set.  Define $\ell(u) = 1$, $\ell(v) = p$, and $\ell(x_{i,j}) = (k-2)(i-1) + j + 1$ for $i = 1, 2, \ldots, n$ and $j = 1, 2, \ldots, k-2$. The labels of all pairs of adjacent vertices are consecutive except for those involving $u$ or $v$. Because $\ell(u) = 1$ and $\ell(v)=p$, all pairs of adjacent vertices have relatively prime labels.

    Next, we assign the edge labels. Consider the sequence of edges created by following the vertices $u, x_{1,1}, x_{1,2}, \ldots, x_{1,k-2}, v, x_{2,k-2}, \ldots, x_{2,1}, u, x_{3, 1}, \ldots, x_{3,k-2},v,\ldots\ldots$, continuing by alternating directions until all pages of the book graph are traversed. For this sequence of edges, assign labels using the following sequences of values:
    \begin{align*}
        \text{If } 3 \;|\; (p+1)&: \ n(k-2) + 2, n(k-2) + 3, \ldots, p - 2, p - 1, p + 2, p + 3, \ldots, 2nk - 3n + 3 \\
        \text{If } 3 \nmid (p+1)&: \ n(k-2) + 2, n(k-2) + 3, \ldots, p-3, p - 2, p + 1, p + 2, \ldots, 2nk - 3n + 3.
    \end{align*}
    The final remaining edge to be labeled is $uv$. If $3 \;|\; (p+1)$, assign $\ell(uv) = p + 1$. Otherwise, assign $\ell(uv) = p - 1$. 
    
    Similarly to Case 1, most of the vertices have at least two consecutively labeled incident edges. In particular, vertex $u$ is incident to the consecutively labeled edges $ux_{2,1}$ and $ux_{3,1}$, while vertex $v$ is incident to the consecutively labeled edges $vx_{1,k-2}$ and $vx_{2,k-2}$. Each vertex $x_{i,j}$ has two incident edges labeled consecutively, except possibly for a single vertex where $p$ was skipped in the edge label sequence. This vertex has incident edges labeled $p - 1$ and $p + 2$ if $3 \;|\; (p + 1)$, or $p - 2$ and $p + 1$ if $3 \nmid (p + 1)$. In each case, the assumption implies that neither label is divisible by 3, so we have $\gcd(p - 1, p + 2) = 1$ or $\gcd(p - 2, p + 1) = 1$. Thus the labels of the incident edges have gcd of 1 for all vertices.

    In both of the above cases, we have provided a total prime labeling. Therefore, $B_k^n$ is total prime for all $k \geq 3$ and $n \geq 3.$
\end{proof}

\section{Non-Prime Cyclic Graphs}\label{sec: nonprime}

We now turn our attention to classes of graphs which are not prime. The primary means for showing a prime labeling does not exist is through its independence number, which is the size of the largest set of vertices where no pair of vertices is adjacent. In order for a prime labeling to be possible, even labels must be placed on independent vertices. Therefore, as first stated in~\cite{FH}, if the independence number of a graph $G$ is less than $\lfloor|V|/2\rfloor$, then $G$ is not prime.

\textit{Complete graphs} on $n$ vertices, denoted $K_n$, have an independence number of 1 since all vertices are adjacent. Hence $K_n$ is not prime for all $n\geq 4$. We will show that it is total prime in these cases, although recall that when $n=3$, $K_3=C_3$ and is not total prime. An example of $K_6$ with a total prime labeling is displayed in Figure~\ref{completegraph}.

\begin{figure}[h]
    \centering
\includegraphics[scale=.6]{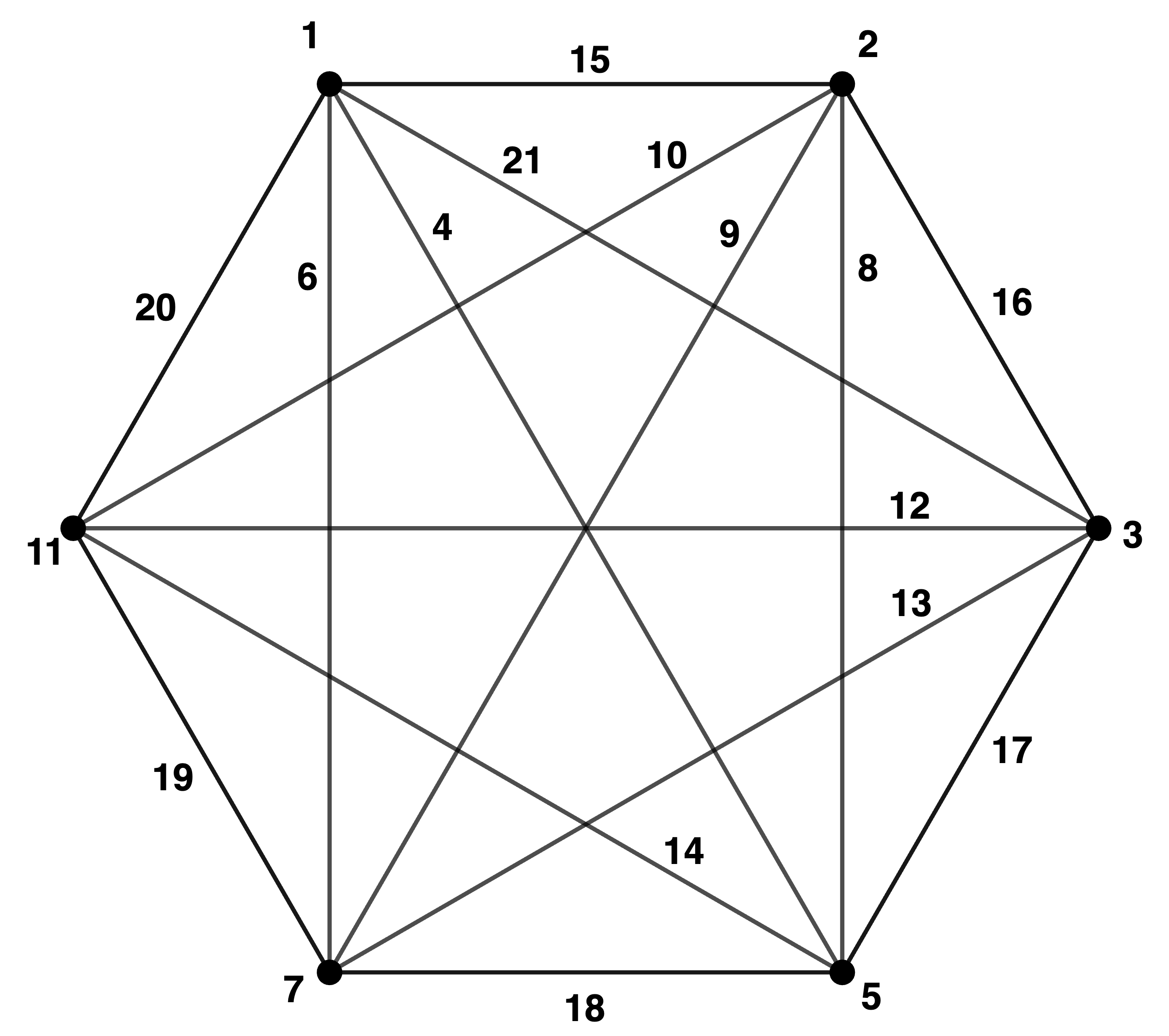}
    \caption{The complete graph $K_6$ with a total prime labeling}
    \label{completegraph}
\end{figure}

Our result regarding complete graphs will rely on an inequality derived from the Prime Number Theorem. Consider the prime counting function $\pi(x)$, which is defined as the number of primes less than or equal to a real number $x$. We will utilize a lower bound for this function~\cite{RS2}, particularly that 
\begin{equation}\label{pi of x}
    \pi(x)>\frac{x}{\ln(x)} \text{ for all } x\geq 17.
\end{equation}
 
\begin{theorem}
    The complete graph $K_n$ is total prime for all $n\geq 4$.
\end{theorem}
\begin{proof}
    We first observe that $K_n$ has $n$ vertices and $\frac{n(n-1)}{2}$ edges for a total of $|V\cup E|=\frac{n^2+n}{2}$. We aim to create a labeling $\ell: V\cup E\rightarrow \{1,2,\ldots,\frac{n^2+n}{2}\}$. Call the vertices $v_1,v_2,\ldots, v_n$.

    We first label a subset of the edges which form a Hamiltonian cycle and chord with the largest $n+1$ labels, where we note that $\frac{n^2+n}{2}-(n+1)=\frac{n^2-n-2}{2}$. That is, assign the following edge labels:
    $$\ell(v_iv_{i+1})=\frac{n^2-n-2}{2}+i \text{ for } i=1,2,\ldots, n-1;\text{ } \ell(v_1v_n)=\frac{n^2+n}{2}-1\text{; and } \ell(v_1v_3)=\frac{n^2+n}{2}.$$
    As in the proof of Theorem~\ref{Hamiltonian}, the gcd condition on sets of incident edges is satisfied for all vertices, no matter the choice of the remaining edge labels.

    For the vertex labels, we assign $\ell(v_1)=1$ and for $i=2,3,\ldots, n$, $\ell(v_i)=p_{i-1}$ where $p_{i-1}$ is the $(i-1)$st prime number. Since $\gcd(1,p_i)$ and $\gcd(p_i,p_j)=1$ for $i\neq j$, the relatively prime condition on the vertices is also satisfied. 

    It remains to show that $\ell$ is injective. In particular, for our edge and vertex labels to not overlap, we need $\ell(v_n)=p_{n-1}\leq \frac{n^2-n-2}{2}$, or equivalently that $\pi\left(\frac{n^2-n-2}{2}\right)\geq n-1$. We will use the lower bound for $\pi(x)$ from Inequality~(\ref{pi of x}), which assumes that $x\geq 17$. Note that when $n\geq 7$, the expression $\frac{n^2-n-2}{2}$ meets this requirement, leaving the cases of $n=4,5,$ and $6$ to be considered shortly.  
    
    Then we have
    $$\pi\left(\frac{n^2-n-2}{2}\right)>\frac{\frac{n^2-n-2}{2}}{\ln(\frac{n^2-n-2}{2})}=\frac{n^2-n-2}{2(\ln(n^2-n-2)-\ln(2))}.$$
    Using a computer algebra system, one can verify that this expression is greater than $n-1$ for all $n>2.56$. While the $\pi(x)$ bound does not apply for $n=4,5$, or $6$, the necessary inequality of $p_{n-1}\leq \frac{n^2-n-2}{2}$ remains true. To verify this, we observe that
    \begin{itemize}
        \item $p_3=5=\frac{4^2-4-2}{2}$,
        \item $p_4=7<9=\frac{5^2-5-2}{2}$,
        \item $p_5=11\leq14=\frac{6^2-6-2}{2}$.
    \end{itemize}
    Therefore, our assumption of $n\geq 4$ is sufficient to guarantee enough prime numbers are available for the vertex labels without using the largest $n+1$ labels in our set, which were previously assigned to edges.

    All unassigned labels can now be used for the remaining edges outside of the cycle and chord. We already showed both relatively prime conditions for $\ell$ are satisfied. Thus, we have proven that $K_n$ is total prime.
\end{proof}

The \textit{windmill graph}, denoted $K_n^{(m)}$ where $m\geq 1$ and $n\geq 3$, consists of $m$ copies of the complete graph $K_n$ joined at one shared vertex. This graph, in the general setting, has not appeared in the literature on prime labelings or variations such as minimum coprime labelings, although it has been studied for the existence of other graph labelings. Note that the $m=1$ case is simply the complete graph, so we will focus on $m\geq 2$.

When $n=3$, this is known as the friendship graph, which is prime~\cite{MV} and total prime~\cite{RK} for all $m\geq 1$. For larger values of $n\geq 4$, upon observing that the windmill graph $K_n^{(m)}$ has $m(n-1)+1$ vertices, we see that it is not prime for any $m\geq 2$. This is because the largest independent set of vertices consists of one vertex from each of the $m$ complete graphs. Therefore, we have $\lfloor|V|/2\rfloor=\lfloor(m(n-1)+1)/2\rfloor\geq \lfloor(3m+1)/2\rfloor>m$.

Although not prime, we suspect $K_n^{(m)}$ is total prime for all $m\geq 2$ and $n\geq 3$. As an example, see Figure~\ref{windmill} in which $m=3$ and $n=4$. However, this general result remains open. We focus on showing it is true for specific cases for $m$ and $n$ in the upcoming results, beginning with $m=2$. 

\begin{figure}[h]
    \centering
\includegraphics[scale=.75]{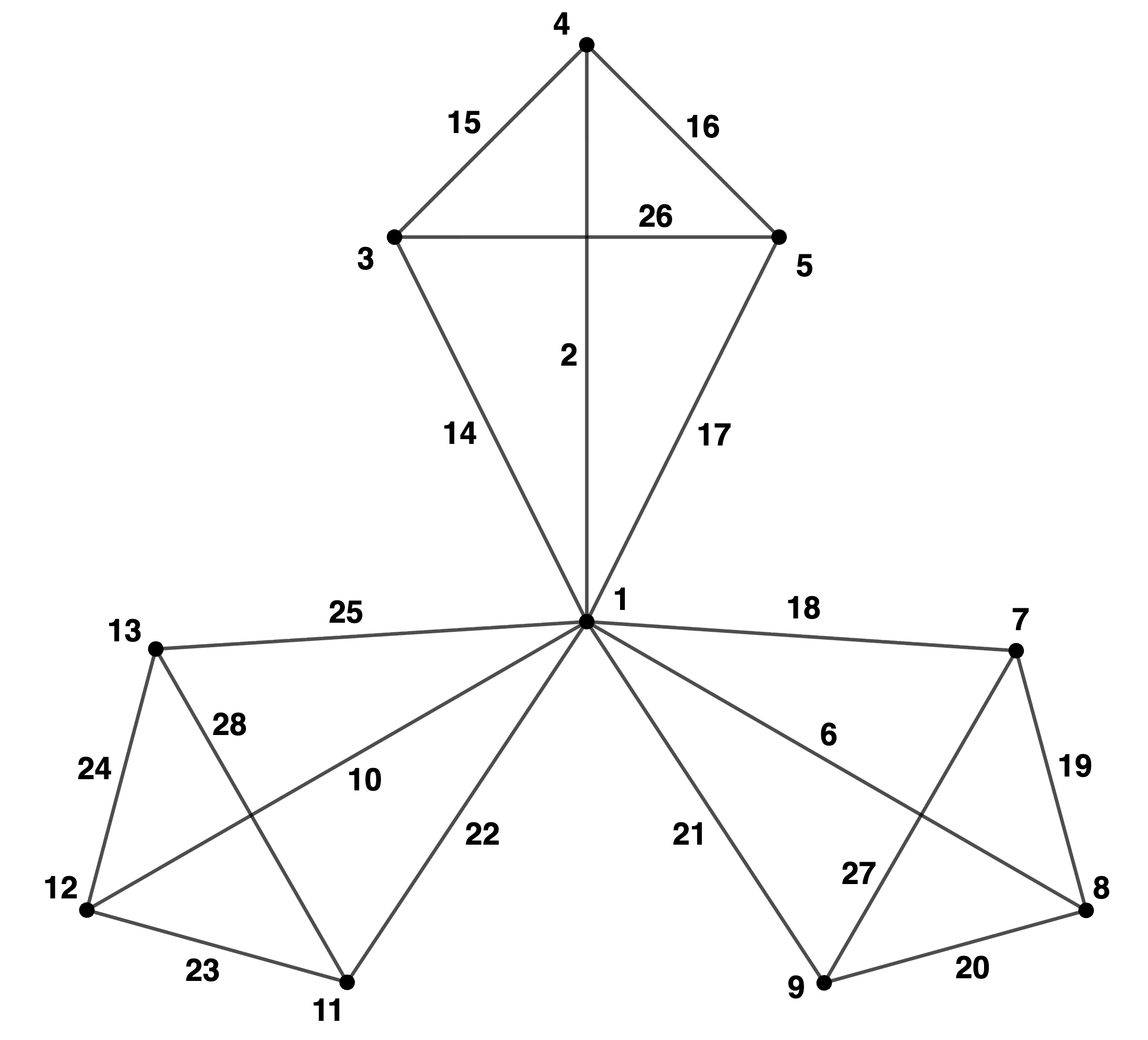}
    \caption{The windmill graph $K_{4}^{(3)}$ with a total prime labeling}
    \label{windmill}
\end{figure}

\begin{theorem}
The windmill graph $K_n^{(2)}$ is total prime for all $n\geq 4$.
\end{theorem}
\begin{proof}
First observe that this graph has $2n-1$ vertices, and the number of edges consists of twice the number for a complete graph $K_n$, or $n(n-1)=n^2-n$. Therefore, our labeling set for a total prime labeling is $\{1,2,\ldots, n^2+n-1$\}.

Call the shared vertex $x$, and the vertices on the two complete graphs $v_1,v_2,\ldots, v_{n-1}$ and $w_1,w_2,\ldots, w_{n-1}$. All pairs of $v_i$ and $v_j$ are adjacent, likewise for $w_i$ and $w_j$, and $x$ is connected by an edge to every other vertex. Define the vertex labels as $\ell(x)=1$, $\ell(v_i)=p_i$ for $i=1,2,\ldots, n-1$, $\ell(w_1)=4$, and $\ell(w_i)=p_{n-2+i}$ for $i=2,3,\ldots, n-1$. Note the largest vertex label is the $p_{2n-3}$, or the $(2n-3)$rd prime number.

For the edge labels, we first label two cycles, one within each complete graph. Assign to the sequence of edges $xv_1, v_1v_2,\ldots, v_{n-2}v_{n-1}, v_{n-1}x$ the integers $n^2-n, n^2-n+1, \ldots, n^2-1$. Then we label the edges $xw_1, w_1w_2, \ldots, w_{n-2}w_{n-1}, w_{n-1}x$ the values $n^2, n^2+1, \ldots, n^2+n-1$, the last of which is the largest integer in our labeling set.

Considering the vertex labels, all adjacent pairs of labels include $\ell(x)=1$, are two distinct prime numbers, or are $\ell(w_1)=4$ with an odd prime number. Hence all of these pairs are relatively prime.

For each of the vertices $v_i$ or $w_i$, the edges incident on that vertex within the two cycles have consecutive labels. Likewise, $\ell(v_{n-1}x)=n^2-1$ and $\ell(xw_1)=n^2$ are also consecutively labeled. Therefore, the gcd condition for incident edges on $x$ is also satisfied.

All that remains to prove $\ell$ is a total prime labeling is to confirm that the vertex and edge labels do not overlap. That is, we need to show $p_{2n-3}<n^2-n$, or likewise $\pi(n^2-n-1)\geq 2n-3.$ For the values $n=4, 5,$ and $6$, we have $p_5=11<12=4^2-4$, $p_7=17<20=5^2-5$, and $p_9=23<30=6^2-6$.

For $n\geq 7$, since $n^2-n-1>17$, we use Inequality~(\ref{pi of x}) to obtain
$$\pi(n^2-n-1)>\frac{n^2-n-1}{\ln(n^2-n-1)}\geq 2n-3,$$
where the second inequality is true for all $n>6.93$. This shows there are sufficient prime numbers to label the vertices which are smaller than the first specified edge label $n^2-n$. Thus, once all remaining labels are distinctly assigned to the other edges, we have a total prime labeling of all windmill graphs with $m=2$. 
\end{proof}

We now fix $n$ for small values, while letting $m\geq 2$ vary. Note that when $m=1$, the windmill is simply the complete graph, which we already investigated.

\begin{theorem}
    The windmill graph $K_4^{(m)}$ is total prime for all $m\geq 2$.
\end{theorem}
\begin{proof}
    The graph $K_4^{(m)}$ has $3m+1$ vertices and $6m$ edges. Let $w$ be the center vertex, and assume $v_{i, 1}$, $v_{i, 2}$, and $v_{i, 3}$ are the vertices of the $i$th copy of $K_4$.

    Define a labeling $\ell: V \cup E \rightarrow \{1, 2, \ldots, 9m+1\}$ where the vertices are labeled as follows:
    $$\ell(w) = 1 \text{ and } \ell(v_{i,j}) = 4i + j - 2, \text{ for } i = 1, 2, \ldots, m \text{ and } j = 1, 2, 3.$$
    To label the edges, first consider the trail formed by vertices $$w, v_{1,1}, v_{1,2}, v_{1,3}, w, v_{2,1}, v_{2,2}, v_{2,3}, w, \ldots, v_{m,1}, v_{m,2}, v_{m,3}, w.$$
    Consecutively label the edges in this trail with the values $4m+2$ through $8m+1$.

    Observe that the labels of all pairs of adjacent vertices are relatively prime. First, we know that $\gcd(\ell(v_{i,j}), \ell(v_{i,j+1})) = 1$ for $j = 1, 2$ because they are consecutive. Next, we have $\gcd(\ell(v_{i,1}), \ell(v_{i,3})) = \gcd(4i-1, 4i+1) = 1$. Lastly, because $\ell(w) = 1$, the condition holds for all pairs of adjacent vertices that include $w$.

    The consecutively labeled closed trail includes all vertices, including $w$ multiple times, eliminating the need for a separate chord used in the Hamiltonian cycle approach.  Therefore, each vertex has two incident edges with consecutive labels, so the gcd of the labels of incident edges is 1 for all vertices.

    Notice that the aforementioned trail consists of $4m$ edges, so there are $|E| - 4m = 6m - 4m=2m$ edges that remain to be labeled, specifically those of the form $wv_{i,2}$ and $v_{i,1}v_{i,3}$. Because we skipped $m$ values in the vertex labeling and the $m$ values from $8m+2$ to $9m+1$ remain unassigned, we can arbitrarily assign these $2m$ values to these remaining $2m$ edges. Thus, we have a total prime labeling of $K_4^{(m)}$.
\end{proof}

\begin{theorem}
    The windmill graph $K_5^{(m)}$ is total prime for all $m\geq 2$.
\end{theorem}
\begin{proof}
    The graph $K_5^{(m)}$ has $4m+1$ vertices and $10m$ edges. Let $w$ be the center vertex and $v_{i,j}$ for $1\leq i\leq m$ and $1\leq j\leq 4$ be the four remaining vertices in each of the $m$ copies of $K_5$.

    Define $\ell: V\cup E\rightarrow \{1,2,\ldots, 14m+1\}$ where we first define the vertex label $\ell(w)=1$. Then for $i=1,2,\ldots, m$, let $\ell(v_{i,1})=6i-3$, $\ell(v_{i,2})=6i-2$, $\ell(v_{i,3})=6i-1$, and $\ell(v_{i,4})=6i+1$.   

    We label the edges in order along the closed trail $$w,v_{1,1}, v_{1,2}, v_{1,3},v_{1,4}, w, v_{2,1}, \ldots, w, v_{m,1}, v_{m,2}, v_{m,3},v_{m,4}, w$$ using the $5m$ labels from $9m+2$ to $14m+1$. Note the largest vertex label was $6m+1$, making~$\ell$ injective thus far. The remaining unused labels can then be assigned arbitrarily to the other edges.

    To confirm all adjacent vertices have relatively prime labels, we first notice this is true for all pairs that include $\ell(w)=1$. Pairs of the form $v_{i,1},v_{i,2}$ and $v_{i,2},v_{i,3}$ have consecutive labels, while pairs $v_{i,1},v_{i,3}$ and $v_{i,3},v_{i,4}$ are labeled by consecutive odd integers. Finally, $v_{i,1}$ and $v_{i,4}$ have odd labels that differ by 4, while the labels $\ell(v_{i,2})$ and $\ell(v_{i,4})$ have a difference of 3 with neither being a multiple of 3. Thus, all adjacent pairs of vertex labels are relatively prime.

    The edges in the previously described trail pass through every vertex, including multiple times through the endpoint of $w$. This implies that all vertices have two or more incident edges that are labeled consecutively, making the gcd of the labels of their incident edges equal 1. Therefore, $\ell$ is a total prime labeling. 
\end{proof}

\begin{theorem}
    The windmill graph $K_{6}^{(m)}$ is total prime for all $m \geq 2$.
\end{theorem}
\begin{proof}
    The graph $K_6^{(m)}$ has $5m + 1$ vertices and $15m$ edges. Let $w$ be the center vertex and $v_{i,j}$ for $1 \leq i \leq m$ and $1 \leq j \leq 5$ be the remaining five vertices of the $i$th copy of $K_6$.

    We define $\ell : V \cup E \rightarrow \{1, 2, \ldots, 20m + 1\}$, where the first vertex label is $\ell(w) = 1$. For each copy of $K_6$, we label the remaining five vertices according to the following distinct cases:
    \begin{itemize}
        \item If $3 \;|\; (10i - 7),$ assign labels $10i - 5$, $10i - 3$, $10i - 2$, $10i - 1$, and $10i + 1$.
        \item If $3 \;|\; (10i - 5),$ assign labels $10i - 7$, $10i - 4$, $10i - 3$, $10i - 1$, and $10i + 1$.
        \item If $3 \;|\; (10i - 3),$ assign labels $10i - 7$, $10i - 5$, $10i - 4$, $10i - 1$, and $10i + 1$.
    \end{itemize}
    Note that the largest vertex label of the graph is $10m + 1$.

    We assign edge labels consecutively along the trail $$w, v_{1, 1}, v_{1, 2}, v_{1, 3}, v_{1, 4}, v_{1, 5}, w, v_{2, 1}, \ldots, w,  v_{m, 1}, v_{m, 2}, v_{m, 3}, v_{m, 4}, v_{m, 5}, w$$ using the labels $14m + 2$ through $20m + 1$. Because the largest vertex label is $10m + 1$ and the smallest edge label is $14m + 2$, the unassigned values left in $\{1, 2, \ldots, 20m + 1\}$ can be assigned to the remaining edges while ensuring that $\ell$ is an injection.

    The vertex $w$ is labeled as $1$, so the labels of any pair of adjacent vertices that include $w$ are relatively prime. In each copy of $K_6$, the remaining five vertices are labeled with one of the three sets of values listed above. Within each case, all labels have differences with prime factors of 2, 3, or 5. Each set contains exactly one even label and at most one label divisible by 5, namely $10i+5$, so no pair shares a common factor of 2 or 5. Also, the assumption for the cases of which value is a multiple of 3 guarantees that at most one label in each $K_6$ has a factor of 3. Finally, no two labels differ by 7 or by more than 8, so no pair of labels can have a common prime factor of at least 7.

    As in the last windmill graph proof, the closed trail that has consecutive labels on at least two incident edges for each vertex implies that $\ell$ is a total prime labeling of $K_6^{(m)}$.
\end{proof}

A \textit{coprime labeling} of a graph $G$ of order $n$ is an injective vertex labeling using integers from $\{1, 2, \ldots, k\}$, for some $k \geq n$, in which labels of adjacent vertices are relatively prime. The smallest integer $k$ for which a coprime labeling of $G$ is possible is called the \textit{minimum coprime number}, denoted $\pr(G)$. A coprime labeling with a maximum label of $\pr(G)$ is a \textit{minimum coprime labeling}~\cite{BDH}.

Minimum coprime numbers have been determined for certain cases of graphs that are known to not be prime, such as prisms~\cite{AF2} and powers of paths and cycles~\cite{AF}. This concept is useful for developing total prime labelings, as a coprime labeling can be extended to include edges. In particular, if the minimum coprime number of a Hamiltonian graph $G$ is sufficiently small relative to the size of $G$, then a total prime labeling of $G$ can be constructed.

\begin{theorem}\label{HamiltonianMCL}
    Let $G$ be a Hamiltonian graph with $|V|=n$, $|E|=m$, and $\Delta\geq 3$. If $\pr(G)\leq m-1$, then $G$ is total prime.
\end{theorem}
\begin{proof}
    Let $f:V\rightarrow \{1,2,\ldots, \pr(G)\}$ be a minimum coprime labeling of $G$. Following a technique similar to the process in Theorem~\ref{Hamiltonian}, we extend $f$ to a total prime labeling $\ell: V \cup E \rightarrow \{1, 2, \ldots, m+n\}$ by first assigning $\ell(v)=f(v)$ for all $v\in V$. For the $n+1$ edges in the Hamiltonian cycle and chord that exist since $\Delta\geq 3$, we assign the labels $\pr(G)+1, \pr(G)+2,\ldots, \pr(G)+n+1$. By the assumption that $\pr(G)\leq m-1$, $\pr(G)+n+1\leq m-1+n+1=m+n$, so the largest edge label is in the labeling set $\{1,2,\ldots, m+n\}$. Arbitrarily assign any unused labels to the remaining edges. Since $f$ is a coprime labeling and every vertex has incident edges labeled consecutively by the cycle and chord, $\ell$ is a total prime labeling.
\end{proof}

The previous theorem applies to many classes of non-prime graphs to show they are total prime. In the following corollary, we focus first on stacked prisms $Y_{m,n}$, which are a generalization of the prism graph formed by the Cartesian product $C_m \square P_n$. We show that they are total prime when $m = 3$ and $m = 5$, while later investigating the $m=4$ case. 

We also consider powers of path and cycle graphs. For a graph $G$ and a positive integer $k$, the $k$th power of $G$ is obtained by adding edges between all pairs of vertices whose distance in~$G$ is at most $k$. We particularly focus on squares and cubes of paths and cycles, denoted $P_n^2$, $P_n^3$, $C_n^2$, and $C_n^3$, and an example of the squared path $P_8^2$ is shown with a total prime labeling in Figure~\ref{pathsquaredgraph}.

\begin{figure}[h]
    \centering
\includegraphics[scale=0.4]{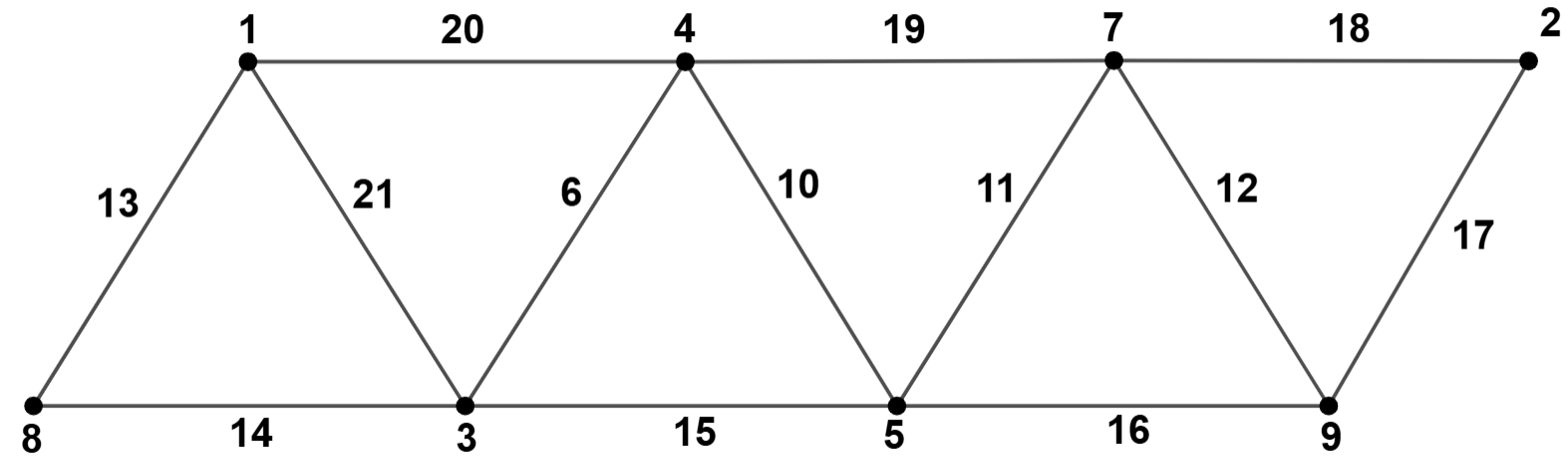}
    \caption{The squared path $P_8^2$ with a total prime labeling}
    \label{pathsquaredgraph}
\end{figure}

\begin{corollary}
The following classes of graphs are total prime: 
    \begin{enumerate}[(a)]
        \item Stacked triangular prisms, $Y_{3,n}$, for $n\geq 2$
        \item Stacked pentagonal prisms, $Y_{5,n}$, for $n\geq 2$
        \item The square of a path, $P_n^2$, for $n\geq 4$
        \item The cube of a path, $P_n^3$, for $n\geq 5$
        \item The square of a cycle, $C_n^2$, for $n\geq 6$ 
        \item The cube of a cycle, $C_n^3$, for $n\geq 8$
    \end{enumerate}
\end{corollary}
\begin{proof}
    One can verify for each of the six graphs that they are Hamiltonian and $\Delta\geq 3$. Then to apply Theorem~\ref{HamiltonianMCL}, we need to show $\pr(G)\leq |E|-1$ in each case.
    \begin{enumerate}[(a)]
        \item The stacked triangular prism $Y_{3, n}$ has $6n - 3$ edges. In \cite{AF2}, the minimum coprime number is given by $\pr(Y_{3, n}) = 4n - 1$. Then we have $$4n - 1 \leq (6n - 3) - 1 \text{ for all } n\geq 2.$$

        \item The stacked pentagonal prism $Y_{5, n}$ has $10n - 5$ edges. The minimum coprime number is given in \cite{AF2} as $\pr(Y_{5, n}) = 6n - 1$. This implies $$6n - 1 \leq (10n - 5) - 1 \text{ for all } n\geq 2.$$

        \item The square of a path $P_n^2$ has $2n - 3$ edges and a minimum coprime number given in~\cite{AF}~as \[
        \pr(P_n^2) = \begin{cases}
            4k - 1, & \text{if } n = 3k \text{ or } n = 3k + 1 \\
            4k + 1, & \text{if } n = 3k + 2
        \end{cases}
        \]
        for $n = 6$ or $n \geq 8$. Note that when $n = 4, 5, \text{and } 7$, $P_n^2$ is prime. Then by Theorem~\ref{Hamiltonian}, $P_n^2$ is total prime in these cases. For $n=6$ and $n \geq 8$,
        \[
        \begin{aligned}
            n = 3k: &\;\;\; 4k - 1 \leq (2(3k) - 3) - 1 \text{ for } k \geq 2 \text{ or } n \geq 6, \\
            n = 3k + 1: &\;\;\; 4k - 1 \leq (2(3k + 1) - 3) - 1 \text{ for } k \geq 1 \text{ or } n \geq 4,\\
            n = 3k + 2: &\;\;\; 4k + 1 \leq (2(3k + 2) - 3) - 1 \text{ for } k \geq 1 \text{ or } n \geq 5.
        \end{aligned}
        \]
       
        \item The cube of a path $P_n^3$ has $3n - 6$ edges. The minimum coprime number is \[
        \pr(P_n^3) = \begin{cases}
            6k - 1, & \text{if } n = 4k \text{ or } n = 4k + 1 \\
            6k + 1, & \text{if } n = 4k + 2 \\
            6k + 3, & \text{if } n = 4k + 3
        \end{cases}
        \]
        for $n \geq 6$ \cite{AF}. Note that when $n = 5$, $P_n^3$ is prime, and thus total prime by Theorem \ref{Hamiltonian}. For $n \geq 6$, we have
        \[
        \begin{aligned}
            n = 4k: &\;\;\; 6k - 1 \leq (3(4k) - 6) - 1 \text{ for } k \geq 1 \text{ or } n \geq 4, \\
            n = 4k + 1: &\;\;\; 6k - 1 \leq (3(4k + 1) - 6) - 1 \text{ for } k \geq 1 \text{ or } n \geq 5,\\
            n = 4k + 2: &\;\;\; 6k + 1 \leq (3(4k + 2) - 6) - 1 \text{ for } k \geq 1 \text{ or } n \geq 6, \\
            n = 4k + 3: &\;\;\; 6k + 3 \leq (3(4k + 3) - 6) - 1 \text{ for } k \geq 1 \text{ or } n \geq 7.
        \end{aligned}
        \]

        \item The square of a cycle $C_n^2$ has $2n$ edges. Its minimum coprime number is \[
        \pr(C_n^2) = \begin{cases}
            4k - 1, & \text{if } n = 3k\\
            4k + 1, & \text{if } n = 3k + 1 \\
            4k + 3, & \text{if } n = 3k + 2
        \end{cases}
        \]
        for $n \geq 4$~\cite{AF}. For $n \geq 6$, we obtain
        \[
        \begin{aligned}
            n = 3k: &\;\;\; 4k - 1 \leq 2(3k) - 1 \text{ for } k \geq 0 \text{ or } n \geq 0, \\
            n = 3k + 1: &\;\;\; 4k + 1 \leq 2(3k + 1) - 1 \text{ for } k \geq 0 \text{ or } n \geq 1,\\
            n = 3k + 2: &\;\;\; 4k + 3 \leq 2(3k + 2) - 1 \text{ for } k \geq 0 \text{ or } n \geq 2.
        \end{aligned}
        \]
        
        \item The cube of a cycle $C_n^3$ has $3n$ edges and minimum coprime number \[
        \pr(C_n^3) = \begin{cases}
            6k - 1, & \text{if } n = 4k \\
            6k + 1, & \text{if } n = 4k + 1 \\
            6k + 5, & \text{if } n = 4k + 2 \\
            6k + 7, & \text{if } n = 4k + 3
        \end{cases}
        \]
        for $n \geq 8$~\cite{AF}. This results in the following:
        \[
        \begin{aligned}
            n = 4k: &\;\;\; 6k - 1 \leq 3(4k) - 1 \text{ for } k \geq 0 \text{ or } n \geq 0, \\
            n = 4k + 1: &\;\;\; 6k + 1 \leq 3(4k + 1) - 1 \text{ for } k \geq 0 \text{ or } n \geq 1,\\
            n = 4k + 2: &\;\;\; 6k + 5 \leq 3(4k + 2) - 1 \text{ for } k \geq 0 \text{ or } n \geq 2, \\
            n = 4k + 3: &\;\;\; 6k + 7 \leq 3(4k + 3) - 1 \text{ for } k \geq 0 \text{ or } n \geq 3.
        \end{aligned}
        \]
    \end{enumerate}
    Each graph satisfies the necessary inequality involving its minimum coprime number and the number of edges, so by Theorem~\ref{HamiltonianMCL}, these six graphs are total prime for their specified order~$n$.
\end{proof}

\section{Other Classes of Graphs}\label{sec: other}

A \textit{prism graph} is the result of the Cartesian product of a path on two vertices and a cycle, denoted as $P_2\square C_n$, and it visually forms the edges and corners of a polygonal prism. It is sometimes denoted as $GP(n,1)$ since it is a special case of a generalized Petersen graph. The prism is conjectured to be prime for all even $n$, which was verified for all $n\leq 2500$~\cite{HLYZ}, although $P_2\square C_n$ is never prime for odd $n$~\cite{PG}. We prove that the prism graph is always total prime for $n$ of either parity. An example of $P_2\square C_{12}$ with a total prime labeling is seen in Figure~\ref{prismgraph}.

\begin{figure}[h]
    \centering
\includegraphics[scale=1.2]{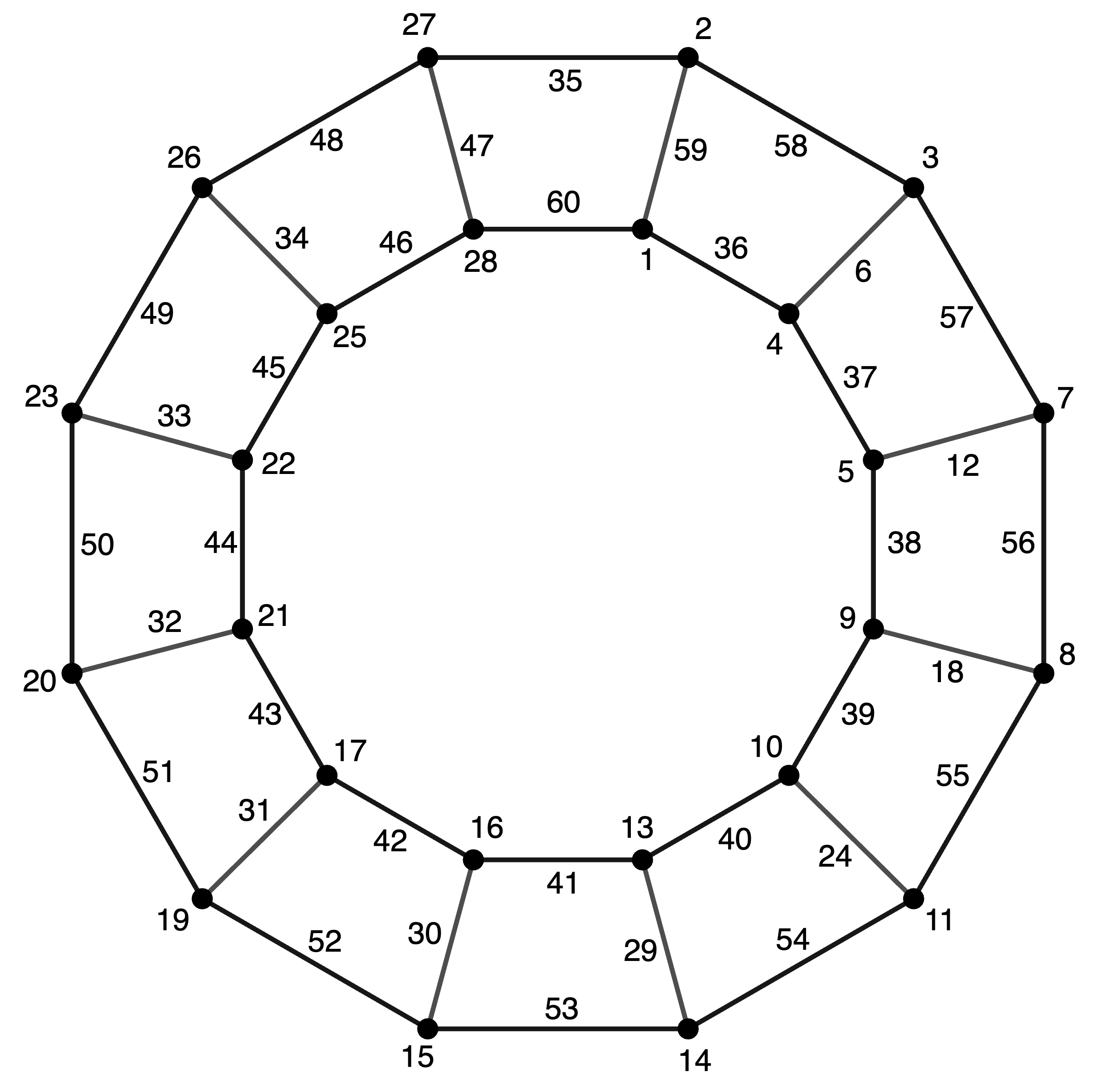}
    \caption{The prism $P_2\square C_{12}$ with a total prime labeling}
    \label{prismgraph}
\end{figure}

\begin{theorem}
    The prism graph $P_2\square C_n$ is total prime for all $n\geq 3$.
\end{theorem}
\begin{proof}
    This graph consists of $2n$ vertices from two $n$-cycles $(u_1,u_2,\ldots, u_n)$ and $(v_1,v_2,\ldots, v_n)$. There are $3n$ edges of the form $u_iu_{i+1}$ and $v_iv_{i+1}$ for $i=1,2,\ldots, n-1$; $u_nu_1$; $v_nv_1$; and $u_iv_i$ for $i=1,2,\ldots, n$. We will create a labeling $\ell: V\cup E\rightarrow \{1,2,\ldots, 5n\}$.
    
    The vertex labels are defined recursively, beginning with the following initial labels:
\begin{itemize}
    \item $\ell(u_1)=1$, $\ell(u_2)=4$, $\ell(u_3)=5$, $\ell(u_4)=9$, $\ell(u_5)=10$
    \item $\ell(v_1)=2$, $\ell(v_2)=3$, $\ell(v_3)=7$, $\ell(v_4)=8$, $\ell(v_5)=11$.
\end{itemize}
For $i> 5$, we define the labels for $u_i$ and $v_i$ by first writing the index in the form $i=5k+j$ for some $k\geq 1$ and $1\leq j\leq 5$. Then assign $\ell(u_i)=12k+\ell(u_j)$, and likewise, $\ell(v_i)=12k+\ell(v_j)$. Essentially, we are assigning each block of five $u_i$ and $v_i$ vertices labels that are shifted by 12 from the corresponding vertices in the previous block. Note in particular that $\ell(u_6)=13$ and $\ell(v_6)=14$.

    There are fifteen edges of the form $u_iu_{i+1}, v_iv_{i+1}$, or $u_iv_i$ for $i=1,2,3,4,$ or $5$. One can observe from the first twelve vertex labels that all fifteen pairs of adjacent labels are relatively prime. Additionally, the differences between these labels are each 1, 2, 3, or 4. 

    Next we consider an adjacent pair $u_iu_{i+1}$, where $i=5k+j$ with $k\geq 1$ and $1\leq j\leq 5$. We have $\ell(u_i)=12k+\ell(u_j)$ and $\ell(u_{i+1})=12k+\ell(u_{j+1})$. Note that $\ell(u_j)$ and $\ell(u_{j+1})$ are relatively prime as one of the initial fifteen adjacencies. Adding a multiple of 12 to two relatively prime numbers cannot create a common prime factor of 2 or 3, although it is possible for the sums to now share a prime factor of 5 or larger. However, since the difference between $\ell(u_j)$ and $\ell(u_{j+1})$ was 1, 2,  3, or 4, the same is true about $|\ell(u_i)-\ell(u_{i+1})|$, so they cannot share any factors greater than 4. Thus, $\gcd(\ell(u_i),\ell(u_{i+1}))=1$. This argument for pairs $v_iv_{i+1}$ and $u_iv_i$ follows analogously.

    The only pairs of adjacent vertices that remain to consider involve the edges $u_nu_1$ and $v_nv_1$. The labels $\ell(u_n)$ and $\ell(u_1)$ are relatively prime since $\ell(u_1)=1$. If $n=5k+j$ where $j=2, 3$, or 5, then $\ell(v_n)$ is odd and therefore relatively prime with $\ell(v_1)=2$. In the cases of $j=1$ and 4, $\ell(v_n)$ is even, so we make the following swaps to the labels in the initial blocks of vertices: $\ell(v_1)=1$, $\ell(u_1)=2$, $\ell(v_2)=4$, and $\ell(u_2)=3$. Now $\ell(v_n)$ and $\ell(v_1)$ are clearly relatively prime, and since $\ell(u_n)$ is odd for $j=1$ and $4$, it has no common factors with the reassigned $\ell(u_1)$. The only new adjacencies within the first block of ten vertices are from the edges $u_2u_3$ and $v_2v_3$, but $\gcd(3,5)$ and $\gcd(4,7)$ are both 1. Thus, for every case of $j=1$ to 5, all adjacent vertex labels are relatively prime.

    For the edge labels, we follow the Hamiltonian cycle and chord approach from earlier proofs. In particular, we consecutively assign to the edges within the Hamiltonian cycle $(u_1, u_2,\ldots, u_n, v_n, v_{n-1},\ldots v_2,v_1, u_1)$ the labels $3n, 3n+1, \ldots ,5n-1$. Then assign $\ell(u_1u_n)=5n$ as the label of a chord. All vertices have incident edges labeled by consecutive integers, including the labels $5n-1$ and $5n$ for edges connecting to $u_1$. Thus, we can assign any remaining labels in $\{1,2,\ldots, 5n\}$ to the edges $v_1v_n$ and $u_iv_i$ for $i=2,3,\ldots, n-1$.

    The labeling $\ell$ satisfies the conditions for a total prime labeling as long as the vertex and edge labels are distinct. The largest vertex label is $12k+2$, $12k+4$, $12k+7$, $12k+9$ or $12k+11$ for the cases of $j=1,2,3,4,$ and $5$, respectively, where $n=5k+j$. The smallest edge label is $\ell(u_1u_2)=3n$, which results in labels of $15k+3$, $15k+6$, $15k+9$, $15k+12$, or $15k+15$. We see that all five cases satisfy $\max\{\ell(v)|v\in V\}<\min\{\ell(e)|e\in E\}$. This makes all the labels in~$\ell$ distinct, proving that it is a total prime labeling.
\end{proof}

We next consider the stacked rectangular prism, $Y_{4,n}$, which is the Cartesian product $C_4 \square P_n$. Unlike the triangular and pentagonal cases, which are not prime, $Y_{4,n}$ does have enough independent vertices for a prime labeling to be possible. However, one has not been introduced in the literature. We will prove it is total prime with Figure~\ref{RectStackedPrism} showing an example with $n=4$.

\begin{figure}[h]
    \centering
\includegraphics[scale=.85]{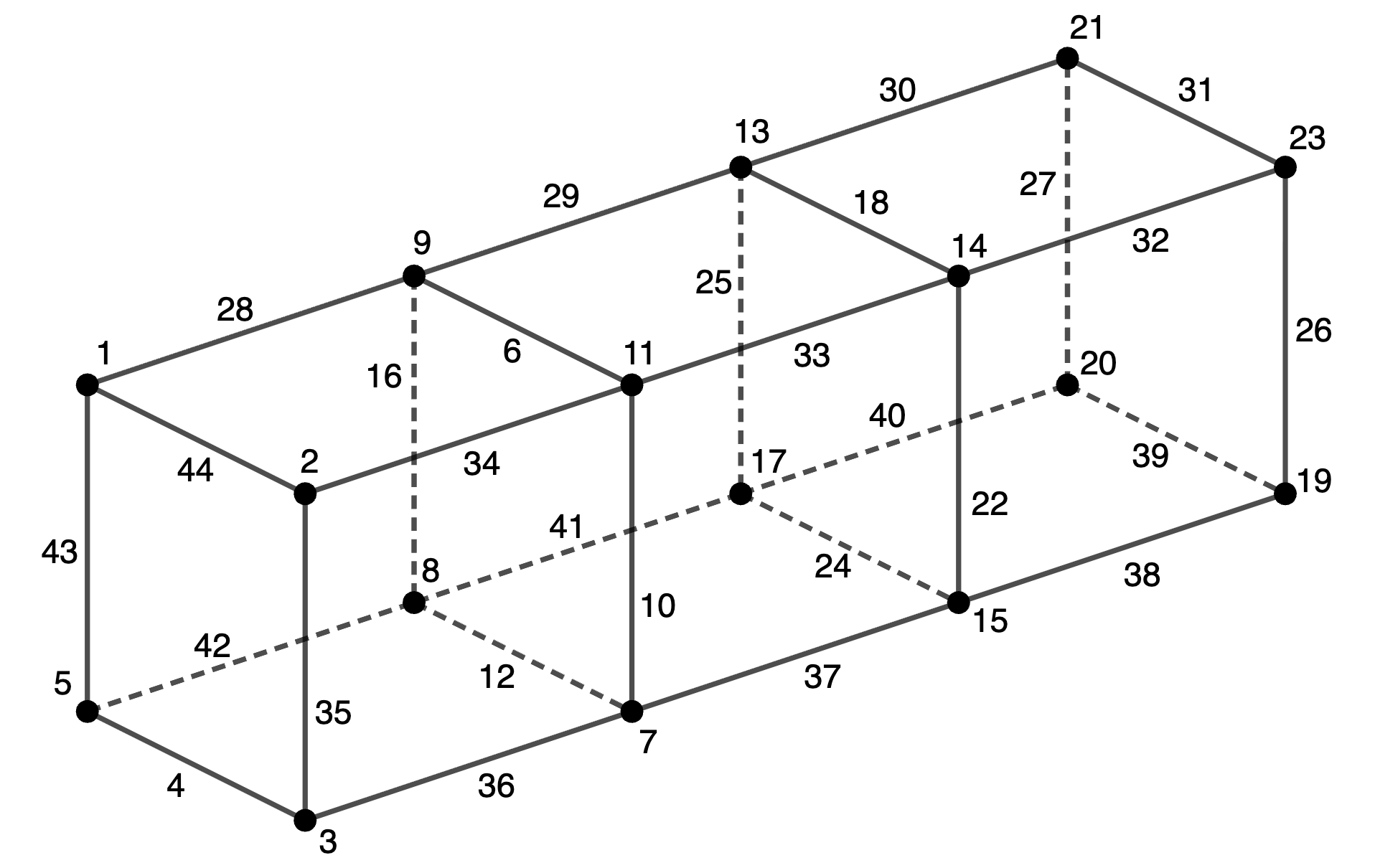}
    \caption{The stacked rectangular prism $Y_{4,4}$ with a total prime labeling}
    \label{RectStackedPrism}
\end{figure}

\begin{theorem}
    The stacked rectangular prism $Y_{4, n}$ is total prime for $n\geq 2$.
\end{theorem}
\begin{proof}
     As a Cartesian product of $C_4$ and $P_n$, the stacked rectangular prism has $4n$ vertices and $8n-4$ edges. Consider each cycle as $(u_i, v_i, w_i, x_i, u_i)$ for $i=1,2,\ldots, n$, with the remaining edges from the four paths such as $u_1,u_2,\ldots, u_n$. We will create a labeling $\ell : V \cup E \rightarrow \{1, 2, \ldots, 12n-4\}$.
     
    We define the vertex labels recursively, where the base labels are as follows:
     $$\ell(u_1) = 1,\; \ell(v_1) = 2, \;\ell(w_1) = 3, \;\ell(x_1) = 5,$$
     $$\ell(u_2) = 9, \;\ell(v_2) = 11,\; \ell(w_2) = 7, \;\ell(x_2) = 8.$$
Each index $i \geq 3$ can be written as $i = 2k + j$ for some $k \geq 1$ and $1 \leq j \leq 2$. We define the remaining vertex labels as 
    \begin{align*}
        \ell(u_i) &= 12k + \ell(u_j), \hspace{.5cm} \ell(v_i) = 12k + \ell(v_j) \\
        \ell(w_i) &= 12k + \ell(w_j), \hspace{.5cm} 
        \ell(x_i) = 12k + \ell(x_j).
    \end{align*}

    There are sixteen pairs of adjacent vertices involving the initial eight vertices. One can observe that these satisfy the relatively prime condition given the explicitly defined labels and the four vertices with $i=3$. Note that the difference between any two adjacent vertices is 1, 2, 3, 4, 8, or 9, all powers of 2 or 3. 

    Consider the adjacent vertex pair $u_iu_{i+1}$ for $1 \leq i \leq n-1$, where $i = 2k + j$ for some $k \geq 1$ and $1 \leq j \leq 2$. We have $\ell(u_i) = 12k + \ell(u_j)$ and $\ell(u_{i+1}) = 12k + \ell(u_{j+1})$. Since $u_j$ and $u_{j+1}$ are one of the first sixteen pairs of adjacent vertices, their labels are relatively prime. Shifting them by a multiple of 12 will maintain this condition for $\ell(u_i)$ and $\ell(u_{i+1})$ since their difference was a power of 2 or 3. This reasoning applies analogously to the remaining vertex pairs $v_iv_{i+1}$, $w_iw_{i+1}$, and $x_ix_{i+1}$ for $1\leq i\leq n-1$, and $u_iv_i, u_ix_i, x_iw_i,$ and $w_iv_i$ for $1 \leq i \leq n$. 

    To label the edges, we identify a Hamiltonian cycle and chord within the graph. Using the values $8n - 4, 8n - 3, \ldots, 12n - 5$, consecutively label the cycle $$(u_1, u_2, \ldots, u_n, v_n, v_{n-1}, \ldots, v_1, w_1, w_2, \ldots, w_n, x_n, x_{n-1}, \ldots, x_1, u_1).$$ For the chord, assign $\ell(u_1 v_1) = 12n - 4$. The other edges can be assigned the remaining unused values in $\{1, 2, \ldots, 12n - 4\}$. Using the reasoning from Theorem~\ref{Hamiltonian}, we have that for every vertex, the gcd of the labels of all incident edges is equal to 1.

    All that remains to show is that the vertex labels and edge labels are assigned distinctly. In total, there are $12n - 4$ labels, with the largest $4n + 1$ labels assigned to the edges in the cycle and chord. We must verify that $\max\{\ell(v) | v \in V\} \leq  8n - 5$ is true for all $n \geq 2$. Note that for $n = 2$, the largest vertex label is 11, and $11 \leq 8n - 5 = 11$. For $n \geq 3$, $n$ can be expressed as either $n = 2k + 1$ or $n = 2k + 2$ for some $k \geq 1$.

    When $n = 2k + 1$, we have that $\max\{\ell(v) | v \in V\} = 12k + 5$. Because $12k + 5 \leq 8n - 5 = 16k + 3$ is true for all $k \geq 1$, the edge and vertex labels are distinct when $n = 2k + 1$. For $n = 2k + 2$, $\max\{\ell(v) | v \in V\} = 12k + 11$. As before, we consider $12k + 11 \leq 8n - 5 =  16k + 11$, which is also true for all $k \geq 1$. The labeling is bijective for all $n \geq 2$, and thus $\ell$ is a total prime labeling.
\end{proof}

\section{Trees}\label{sec: trees}

We begin our examination of trees with the particular class of the \textit{bistar}, which is a graph obtained by connecting the centers of two star graphs with an edge. Formally, we define $B_{m,n}$ to consist of two stars, $K_{1,m}$ and $K_{1,n}$, where $m\leq n$. Let $u$ be the center of $K_{1,m}$ which is connected to the vertices $w_1, w_2,\ldots, w_m$, and $v$ be the center of $K_{1,n}$ adjacent to $x_1,x_2,\ldots, x_n$. The edge $uv$ completes the structure of the bistar. See Figure~\ref{Bistar} for an example of a total prime labeling on $B_{4,5}$.

Note that the bistar was claimed to be total prime in~\cite{RK}. However, their labeling fails to be injective, as vertex and edge labels overlap in cases such as $m=4$, $n=6$. We provide the following labeling to confirm that their claim is in fact true for all sizes of bistar. 

\begin{figure}[h]
    \centering
\includegraphics[scale=1]{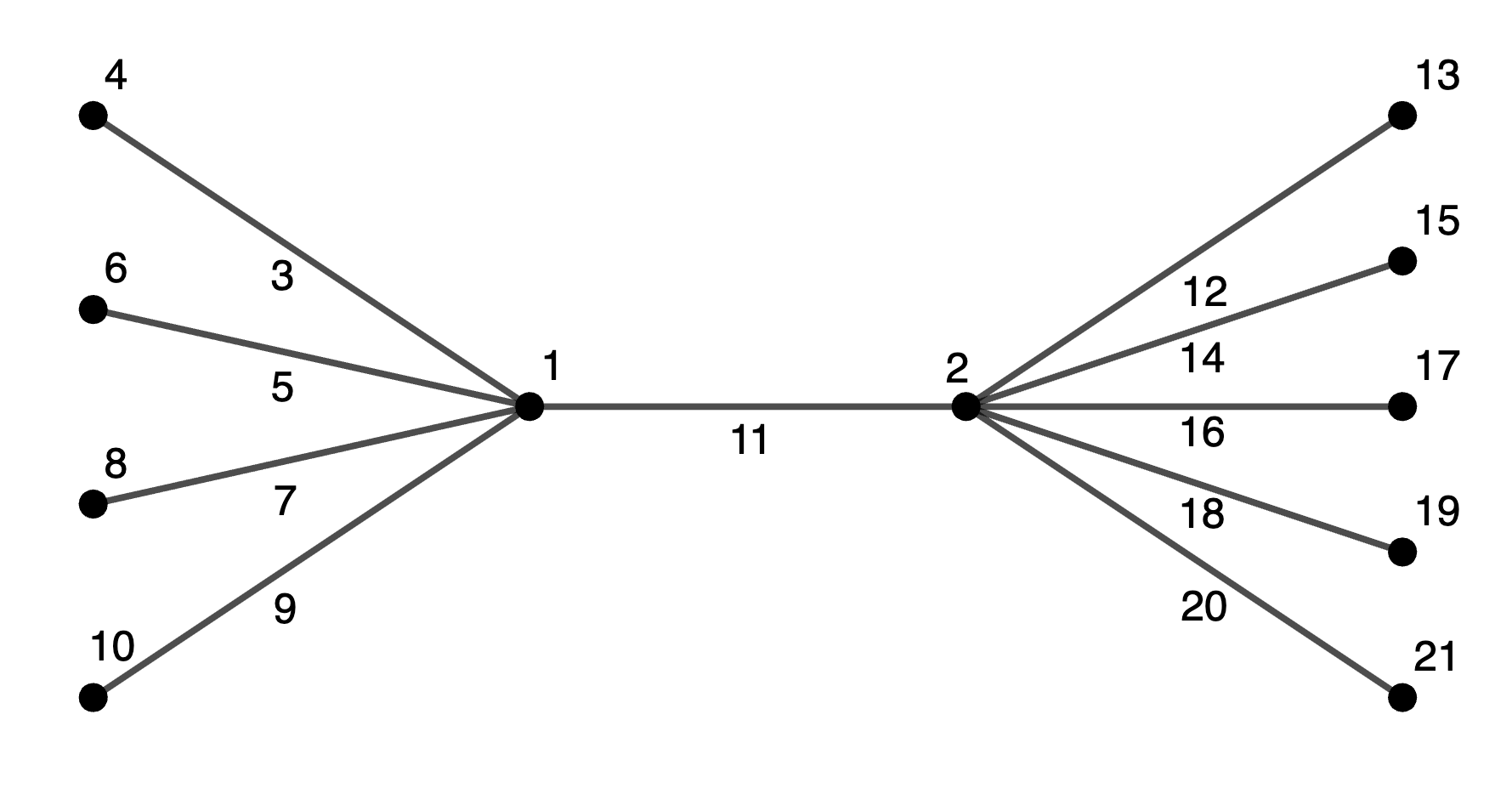}
    \caption{The bistar $B_{4,5}$ with a total prime labeling}
    \label{Bistar}
\end{figure}

\begin{theorem}
    The bistar $B_{m,n}$ is total prime for all $m,n\geq 1$.
\end{theorem}
\begin{proof}
    This graph has $m+n+2$ vertices and $m+n+1$ edges. Then we create a labeling $\ell: V\cup E\rightarrow \{1,2,\ldots, 2m+2n+3\}$ as follows:
    $$    \ell(u)=1, \;\;
        \ell(v)=2,\;\;
        \ell(uw_i)=2i+1,\;\;
        \ell(w_i)=2i+2, \text{ for } i=1,2,\ldots, m,$$
    $$    \ell(uv)=2m+3, \;\; 
        \ell(vx_j)=2m+2j+2,\;\;
        \ell(x_j)=2m+2j+3, \text{ for } j=1,2,\ldots, n.
    $$

    Observe that all adjacent pairs of vertices involve at least one of the centers $u$ and $v$ of the stars. Since $\ell(u)=1$, its label is relatively prime with $\ell(v)$ and all $\ell(w_i)$. Since $\ell(v)=2$ and each $\ell(x_j)$ is odd, these labels are relatively prime as well.

    The only vertices with degree greater than 1 are $u$ and $v$. For edges incident on $u$, we have 
    $$\gcd(\ell(uw_1),\ldots, \ell(uw_m),\ell(uv))=\gcd(3,5,\ldots, 2m+1, 2m+3)=1$$
    since this set of labels contains consecutive odd integers. The edges incident on $v$ result in 
    $$\gcd(\ell(uv),\ell(vx_1),\ldots, \ell(vx_n))=\gcd(2m+3, 2m+4,2m+6,\ldots, 2m+2n+2)=1$$
    because the first two labels are consecutive.

    Since all pairs of adjacent vertices have relatively prime labels, and the gcd of the labels of incident edges is 1 for each internal vertex, this is a total prime labeling of $B_{m,n}$.
\end{proof}

Rather than investigating other individual classes of trees to determine if each is total prime, we take a more general approach and rely on the vast literature of prime labelings of trees. The following result demonstrates how the vertex labels of any prime tree can be extended by labeling its edges to form a total prime labeling. Figure~\ref{CBT} shows how a total prime labeling is constructed for a complete binary tree of depth 4.

\begin{figure}[h]
    \centering
\includegraphics[scale=1]{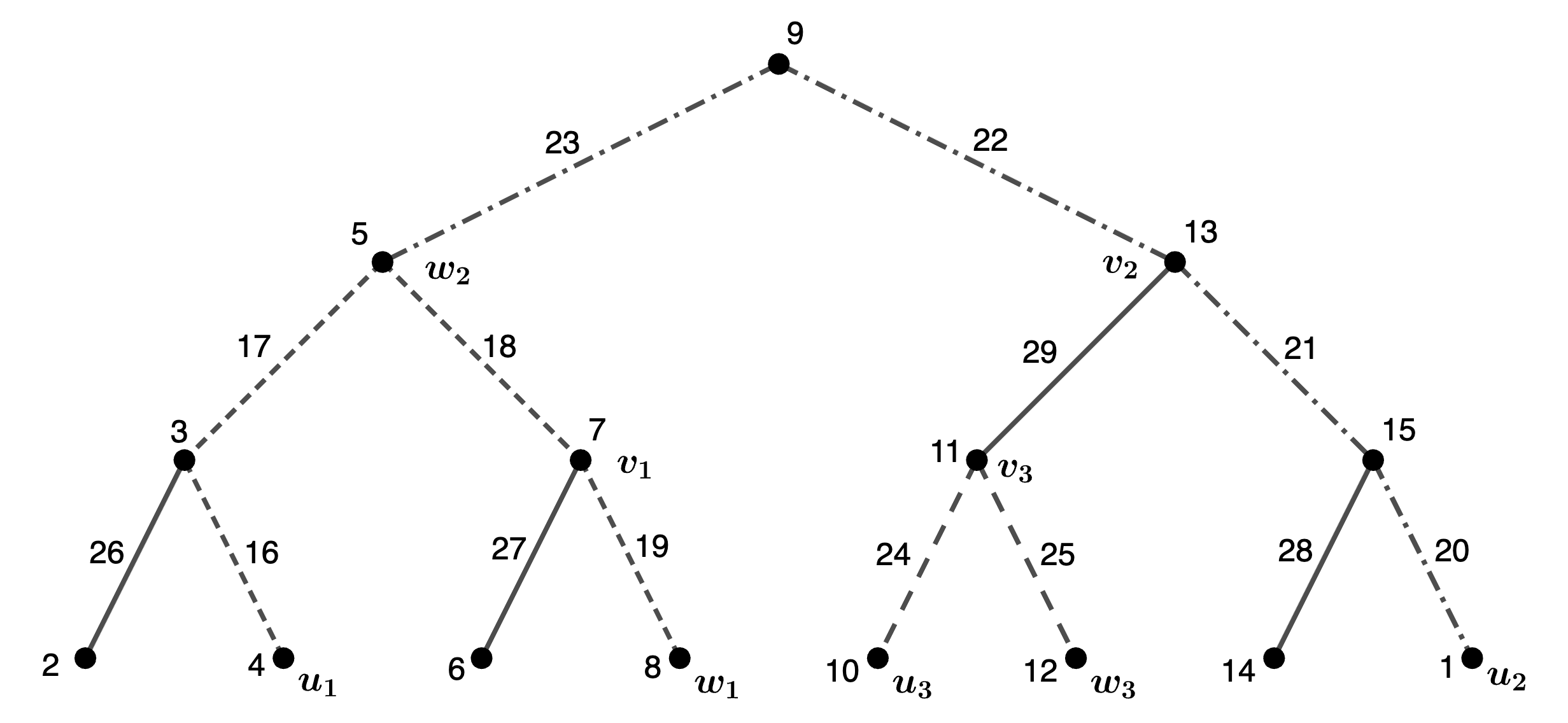}
    \caption{A complete binary tree with 4 levels with a total prime labeling, where vertices are labeled as in~\cite{FH} and edges are labeled using the indicated paths as described in Theorem~\ref{trees}}
    \label{CBT}
\end{figure}

\begin{theorem}\label{trees}
    If $G$ is a tree that is prime, then $G$ is total prime.
\end{theorem}
\begin{proof}

    Let $f: V \rightarrow \{1, 2, \ldots, |V|\}$ be the prime labeling of the vertices. We define a labeling $\ell: V \cup E \rightarrow \{1, 2, \ldots, |V| + |E|\}$ with $\ell(v) = f(v)$. To label the edges, we sequentially label paths that collectively form an edge cover of the set of vertices which have degree at least 2.

    We begin with a vertex $v_1$ with $\deg(v_1) \geq 2$. Since $G$ is a tree, we can create disjoint paths from $v_1$ to two distinct leaves $u_1$ and $w_1$. Joined together, they form a path $P_1 = u_1, \ldots, v_1, \ldots, w_1$. Note that if $G=P_1$, then we are done since paths are known to be total prime~\cite{RK}. 

    Otherwise, we proceed by creating additional paths until all non-leaf vertices are on the interior of one such path. For the path $P_2$, we consider a second non-leaf vertex $v_2$ that does not lie on the path $P_1$. We create a path going out from $v_2$ and ending when we first reach a vertex $u_2$ that is a leaf or is on the previously constructed path. Similarly, we create a distinct path from $v_2$ to $w_2$, which again is the first vertex reached on $P_1$ or is a leaf. Together, these two paths form $P_2 = u_2, \ldots, v_2, \ldots, w_2$, which does not share any edge with $P_1$.

    Continue this process of creating paths $P_i = u_i, \ldots, v_i, \ldots, w_i$ where $v_i$ is a non-leaf vertex not on $P_j$ for $j<i$. Each path $P_i$ may share a vertex with a previous $P_j$, but only if it is the endpoint of $P_i$, resulting in any edge from $G$ being included in at most one path. We conclude when all non-leaf vertices are included in $P_1 \cup P_2 \cup \ldots \cup P_k$ for some $k\geq 1$, specifically as internal vertices.

    Let $n_i$ be the length of $P_i$. Consecutively label the edges of each path as follows:
    \begin{align*}
        &P_1: \hspace{2mm} |V| + 1, |V| + 2, \ldots, |V| + n_1, \\
        &P_2: \hspace{2mm} |V| + n_1 + 1, \ldots, |V| + n_1 + n_2, \\
        &\hspace{2mm}\vdots \\
        &P_k: \hspace{2mm} |V| + n_1 + \ldots + n_{k-1} + 1, \ldots, |V| + n_1 + \ldots + n_k.
    \end{align*}
    The remaining edges in $E(G)$ that do not lie on a path $P_i$ can be distinctly assigned the remaining labels $|V| + n_1 + \ldots + n_k+1, |V| + n_1 + \ldots + n_k+2, \ldots, |V|+|E|$ in any order.

    For any vertex $x$ with $\deg(x) \geq 2$, the gcd of the labels of its incident edges is 1 since it is an internal vertex of a path $P_i$, implying two of its incident edges are labeled consecutively. Additionally, the vertex labeling $f$ was given to be prime, so all pairs of adjacent vertices have relatively prime labels. Therefore, G is total prime.
\end{proof}

The following result is a collection of prime trees for which we can construct a total prime labeling. This is by no means an exhaustive list, as much other work has been done on prime labelings of trees. Furthermore, all trees were conjectured to be prime~\cite{TDH}, a result which, if proven, would carry over to total prime labelings.

\begin{corollary}
    The following classes of trees are total prime:
    \begin{enumerate}[(a)]
        \item All spider graphs 
        \item The $t$-toed caterpillar for any number of toes $t\geq 1$ and length $n\geq 1$
        \item Complete binary trees $T_2(n)$ of any number of levels $n\geq 1$ 
        \item Palm trees $PT_{n,k}$ for any $n,k\geq 1$ 
        \item Banana trees of any size 
        \item All trees with at most 50 vertices 
    \end{enumerate}
\end{corollary}
\begin{proof}
    These trees are total prime as a direct application of Theorem~\ref{trees} based on the following results for prime labelings: (a)~\cite{LWY}, (b)~\cite{TDH}, (c)~\cite{FH}, (d)~\cite{RS}, (e)~\cite{RS}, and (f)~\cite{Pik}.
\end{proof}

\section{Non-Total Prime Graphs}\label{sec: non-TPL}  

The only graph previously shown to not be total prime is the odd cycle $C_{2k+1}$ for any $k\geq 1$~\cite{RK}. We will introduce larger classes of graphs that do not admit a total prime labeling, all of which include an odd cycle as a subgraph. These will consist of the union of disjoint graphs, denoted $G\cup H$, with the vertex set $V(G)\cup V(H)$ and edge set $E(G)\cup E(H)$.

Our proof involving the union of cycles will rely on a connection to prime labelings of similar unions, which were investigated in~\cite{DLM}. Specifically, any union of cycles is not prime if at least two of the cycles are of odd length.

\begin{theorem}
    The union of $n$ cycles $\bigcup_{k=1}^n C_{i_k}$ is not total prime if at least one cycle length $i_k$ is odd.
\end{theorem}
\begin{proof}
    Assume $G=\bigcup_{k=1}^n C_{i_k}$ is a union in which at least one $i_k$ is odd. Suppose for the sake of a contradiction that $\ell$ is a total prime labeling of $G$.
    
    Consider a single cycle within the union, say $C_{m}$ on the vertices $v_1,v_2,\ldots, v_{m}$. The total prime labeling $\ell$ would require that pairs of labels $\ell(v_j)$, $\ell(v_{j+1})$ be relatively prime for all $j=1,2,\ldots, m-1$, in addition to $\ell(v_m)$, $\ell(v_1)$. Since every vertex is of degree 2, the total prime condition on the edge labels also results in relatively prime pairs $\ell(v_jv_{j+1})$, $\ell(v_{j+1}v_{j+2})$ for $j=1,2,\ldots, m-2$, along with pairs $\ell(v_{m-1}v_{m}),\ell(v_{m}v_{1})$ and $\ell(v_{m}v_{1}),\ell(v_{1}v_{2})$.

    Now consider a graph $C_{m}\cup C_{m}$ on vertices $u_1,u_2,\ldots, u_{m}$ and $x_1,x_2,\ldots, x_{m}$. Define a vertex labeling $f$ of this union graph by $f(u_j)=\ell(v_j)$ and $f(x_j)=\ell(v_jv_{j+1})$ for $j=1,2,\ldots, m$ with the exception of $f(x_{m})=\ell(v_{m}v_1)$. Since the adjacent vertices in the second cycle mirror the pairs of incident edges of the graph $C_m$ within $G$, the assumption that $\ell$ is total prime implies the labeling $f$ satisfies the relatively prime condition for a prime labeling of $C_m\cup C_m$.

    Define the graph $H$ to be a union of cycles with each component cycle of $G$ included twice in the union. That is, let $H=G\cup G$. Note that $H$ has at least two odd cycles since there was an odd cycle within $G$. The labeling $f$ defined above for the two copies of $C_{i_k}$, for each $i_k$, would satisfy the relatively prime condition. Additionally, it would be injective since $|V(H)|=|V(G)\cup E(G)|$, making $f$ a prime labeling.

    This is a contradiction because, as shown in~\cite{DLM}, a union of cycles cannot be prime if there are at least two odd cycles. Therefore, a total prime labeling cannot exist for this union of cycles.
\end{proof}

For our final result, we demonstrate that any graph $G$ will become not total prime when a disjoint union is applied with sufficiently many 3-cycles. We use the notation $mC_3$ to mean the union of $m$ disjoint copies of $C_3$. The reasoning behind our proof can be applied to larger odd cycles as well. Additionally, a tighter lower bound for $m$ can be determined at a lower value once a specific graph $G$ is chosen. An example of this is seen in Figure~\ref{Union-of-odd-cycles}, which shows how only $m=2$ is sufficient to union with $K_5$ to make it not total prime.

\begin{figure}[h]
    \centering
\includegraphics[scale=.55]{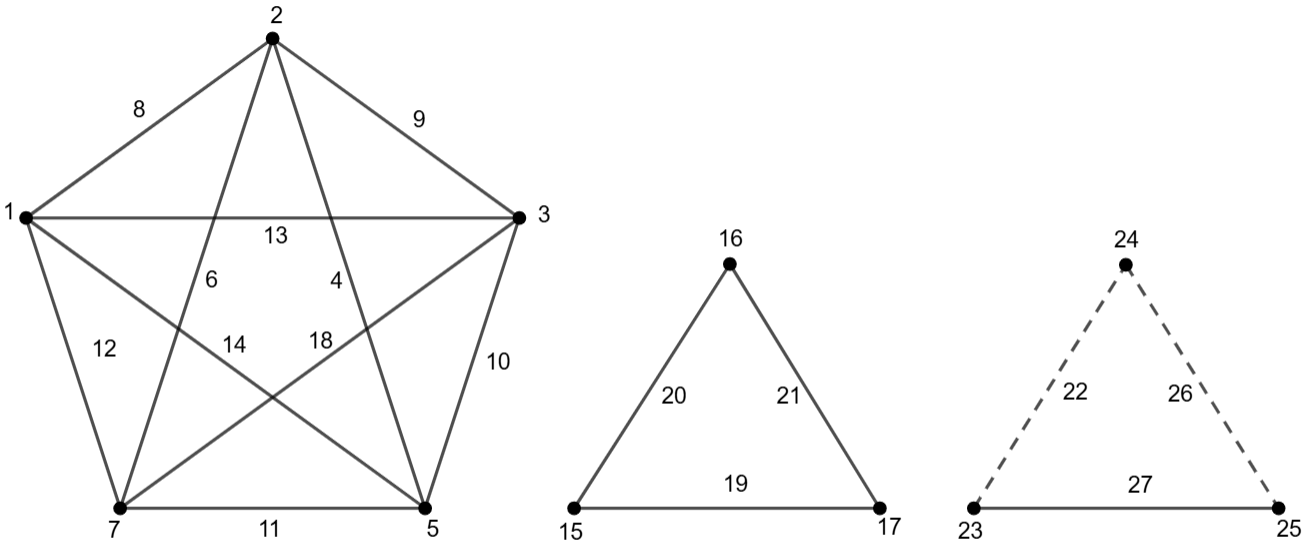}
    \caption{The graph $K_5\cup C_3\cup C_3$ with a labeling that fails to be total prime}
    \label{Union-of-odd-cycles}
\end{figure}

\begin{theorem}
   For all graphs $G$ of order $n\geq 2$, $G\cup mC_3$ is not total prime if $m>\frac{n(n+1)}{2}$.
\end{theorem}
\begin{proof}
Let $G$ be a graph on $n$ vertices, and assume $m>\frac{n(n+1)}{2}$. The graph $G\cup mC_3$ has $|V|=3m+n$ vertices, and while the number of edges $|E|$ is not known, the maximum number occurs when $G=K_n$. Therefore, $|E|\leq 3m+\frac{n(n-1)}{2}$. We focus on the parity of the integers of a labeling $\ell$ on $V\cup E$, where $|V\cup E|\leq 6m+\frac{n(n+1)}{2}$.

For each of the $m$ copies of $C_3$, at most one even label can be used on a vertex to avoid a common factor of 2 between adjacent vertex labels. Likewise, at most one even edge label can be assigned on each cycle, else the two incident edge labels for a vertex would have a gcd of at least 2. Hence, we can assign at most $2m$ even labels on the vertices and edges of $mC_3$, and therefore at least $4m$ odd labels are necessary.

Our labeling set contains at most the integers from 1 to $6m+\frac{n(n+1)}{2}$. Then, using our assumptions of $n\geq 2$ and $m>\frac{n(n+1)}{2}$, the number of odd labels is at most
\begin{align*}
    \left\lceil\frac{6m+\frac{n(n+1)}{2}}{2}\right\rceil
    &\leq 3m+\frac{n(n+1)}{4}+1\\
    &<3m+\frac{n(n+1)}{2}\\
    &<3m+m\\
    &=4m.
\end{align*}
We have shown there are not enough odd labels in the entire labeling set for the labeling $\ell$ to satisfy the necessary conditions on the adjacent vertices and incident edges of the $mC_3$ components. Thus, the graph $G\cup mC_3$ is not total prime.
\end{proof}

\section{Open Problems}\label{sec: open}

We conclude with some open problems involving total prime labelings. 

\begin{itemize}
    \item Complete bipartite graphs $K_{m,n}$ have only been shown to be prime in the cases of $m=1$ or $2$~\cite{RK}. Is $K_{m,n}$ total prime for all $m,n\geq 1$?
    \item Can our labeling of the windmill graph when $m$ or $n$ is small be extended to show $K_n^{(m)}$ is total prime for all $m\geq 2$ and $n\geq 4$?
    \item While minimum coprime labelings of higher powers of paths and cycles have not been developed, could the additional flexibility from the edge labels allow total prime labelings to be constructed for $P_n^k$ and $C_n^k$ with $k\geq 4$? Likewise, can our labelings be generalized for $Y_{k,n}$ with $k\geq 6$? 
    \item Since the complete graph $K_n$ is total prime for all $n\geq 4$, can a condition relating to the density of a graph be developed to guarantee a graph with a sufficient number of edges is total prime?
\end{itemize}

\end{document}